\documentclass[hidelinks,onefignum,onetabnum]{siamart220329}

\usepackage{lipsum,bm,comment}
\usepackage{amsmath,amsfonts,amssymb}
\usepackage{graphicx,subfigure}
\usepackage{epstopdf,float,tikz}
\usepackage{algorithmic}
\usepackage{multirow,adjustbox}
\usepackage{ctable}
\newcommand{\specialcell}[2][c]{\begin{tabular}[#1]{@{}c@{}}#2\end{tabular}}
\ifpdf
  \DeclareGraphicsExtensions{.eps,.pdf,.png,.jpg}
\else
  \DeclareGraphicsExtensions{.eps}
\fi
\graphicspath{{./figures_cv/}}

\newcommand{\roma}{\mathrm{I}}
\newcommand{\romb}{\mathrm{II}}

\newcommand{\bs}[1]{{#1}}

\newcommand{\mc}[1]{\mathcal{#1}}

\newcommand{\dx}{\,\mathrm{d}x}
\newcommand{\prnt}[1]{\left( #1 \right)}

\newcommand{\norm}[1]{\left\|#1\right\|}

\newcommand{\normL}[2]{\norm{#1}_{L^2\prnt{#2}}}

\newcommand{\normI}[2]{\norm{#1}_{L^\infty\prnt{#2}}}

\newcommand{\Cov}{C_{\mathrm{ov}}}

\newcommand{\Cb}{{{\rm C}}({\rm{D}},{\bm{b}})}

\newcommand{\locv}[3]{{#1}^{#2}_{\mathrm{#3}}}
\newcommand{\PeH}{{\rm{Pe}}_{H,{\bm{b}},\epsilon}}

\newcommand{\noteLi}[1]{{#1}}
\newcommand{\workLi}[1]{{#1}}
\newcommand{\noteLg}[1]{{#1}}
\newcommand{\workLg}[1]{{#1}}
\newsiamremark{remark}{Remark}
\newsiamremark{hypothesis}{Hypothesis}
\crefname{hypothesis}{Hypothesis}{Hypotheses}
\newsiamthm{claim}{Claim}

\headers{WEMSFEM FOR CONVECTION DIFFUSION PROBLEMS}{S. Fu, E. Chung, and G. Li}

\title{Wavelet-based Edge Multiscale Finite Element Methods for Singularly Perturbed Convection-Diffusion Equations\thanks{Submitted to the editors May 3, 2024.
\funding{GL acknowledges the support from Newton International Fellowships Alumni following-on funding awarded by The Royal Society, Young Scientists fund (Project number: 12101520) by NSFC and Early Career Scheme (Project number: 27301921), RGC, Hong Kong. SF's research is supported by startup funding of Eastern institute of technology, Ningbo, NSFC (Project number: 12301514) and Ningbo Yongjiang Talent Programme. The research of EC is partially supported by the Hong Kong RGC General Research Fund (Project numbers: 14305222 and 14304021).}}}

\author{Shubin Fu\thanks{Eastern Institute for Advanced Study, Eastern Institute of Technology, Ningbo, Zhejiang 315200, P. R. China.
  (\email{shubinfu@eias.ac.cn}, \url{http://www.imag.com/\string~ddoe/}).}
\and Eric Chung\thanks{Department of Mathematics, The Chinese University of Hong Kong, Hong Kong Special Administrative Region.  
  (\email{eric.t.chung@gmail.com}, \email{jesmith@fictional.edu}).}
\and Guanglian Li\thanks{Corresponding author. Department of Mathematics, The University of Hong Kong, Pokfulam Road, Hong Kong Special Administrative Region,\email{lotusli@maths.hku.hk}}.}

\usepackage{amsopn}

\ifpdf
\hypersetup{
  pdftitle={An Example Article},
  pdfauthor={D. Doe, P. T. Frank, and J. E. Smith}
}
\fi

\begin{document}

\maketitle

\begin{abstract}
We propose a novel efficient and robust Wavelet-based Edge Multiscale Finite Element Method (WEMsFEM) motivated by \cite{MR3980476,GL18} to solve the singularly perturbed convection-diffusion equations. The main idea is to first establish a local splitting of the solution over a local region by a local bubble part and local Harmonic extension part, and then derive a global splitting by means of Partition of Unity. This facilitates a representation of the solution as a summation of a global bubble part and a global Harmonic extension part, where the first part can be computed locally in parallel. To approximate the second part, we construct an edge multiscale ansatz space locally with hierarchical bases as the local boundary data that has a guaranteed approximation rate both inside and outside of the layers. The key innovation of this proposed WEMsFEM lies in a provable convergence rate with little restriction on the mesh size. Its convergence rate with respect to the computational degree of freedom is rigorously analyzed, which is verified by extensive 2-d and 3-d numerical tests.
\end{abstract}

\begin{keywords}
Multiscale method, Convection-diffusion, Singularly perturbed, Partition of Unity, Hierarchical bases
\end{keywords}

\begin{MSCcodes}
65N30, 65N15, 65N12
\end{MSCcodes}

\section{Introduction}
This paper focuses on solving singularly perturbed convection diffusion equation. Let $D\subset
\mathbb{R}^d$ ($d=1,2,3$) be an open bounded domain with a Lipschitz boundary. Then we seek a function $u\in V:=H^{1}_{0}(D)$ such that
\begin{equation}\label{eq:original}
\begin{aligned}
-\epsilon\Delta u+\bm{b}\cdot\nabla u&=f \quad \text{in } D,\\
u&=0 \quad \text{on } \partial D,
\end{aligned}
\end{equation}
where the force term $f\in L^2(D)$, $0<\epsilon\ll 1$ is a \noteLi{constant} parameter
and the velocity field $\bm{b}\in L^{\infty}(D;\mathbb{R}^d)$. We assume that the velocity field $\bm{b}$ is incompressible, i.e., $\nabla\cdot \bm{b}=0$.
The focus of this paper is on the convection-dominated regime or the singularly perturbed case
with large P\'eclet number $\mathrm{Pe}:=\normI{\bm{b}}{D}/\epsilon$.

Many important applications such as metamaterials, porous media and fluid mechanics involves multiple scales. For example, metamaterials are periodic artificial material with the size of each cell being on a subwavelength scale that possesses desirable properties which are not available for the natural material. Another example arises from singularly perturbed convection-diffusion problem \eqref{eq:original}, whose solution has high oscillation in the form of boundary layer or internal layers. Due to this disparity of scales, the classical numerical treatment becomes prohibitively expensive
and even intractable for many multiscale applications. Nonetheless, motivated by the broad spectrum of practical applications, a large number of multiscale model reduction techniques, e.g., Multiscale Finite Element Methods (MsFEMs),
Heterogeneous Multiscale Methods (HMMs), Variational Multiscale Methods (VMM), flux norm approach, Generalized Multiscale Finite Element Methods (GMsFEMs) and Localized Orthogonal Decomposition (LOD), have been proposed in the literature \cite{MR2721592, chung2023multiscale,MR1979846,egh12,MR1455261,MR1660141, MR3246801} over the last few decades.  \cite{altmann2021numerical,chung2023multiscale,
owhadi2019operator} for recent overviews. They have achieved great success in the efficient and accurate simulation of problems involving multiple scales.

The singularly perturbed convection-diffusion problem \eqref{eq:original} is numerically challenging due to the presence of boundary layers or internal layers with width of $\mathcal{O}(\epsilon)$, which are usually exponential. To obtain a reasonable macroscale solution, one has to resolve all these \workLi{multiple scales hidden in the solution $u$, encoded as the layers or high oscillatory feature,} using a very fine mesh, which is computational extremely infeasible.
To reduce the computational complexity, one category of numerical methods is based upon stabilization, albeit at the loss of accuracy especially near the layers \cite{calo2016multiscale,CHUNG20202336,MR1365381,le2017numerical,
li2017error,MR2454024}. Another category is to incorporate the information of the layers in the form of local multiscale basis functions into the trial space \cite{bonizzoni2022super,CHUNG20202336,park2004multiscale,zhao2023constraint}. These methods usually achieve good accuracy even inside the layers. The objective of this paper is to develop a multiscale method that falls into the second category such that it achieves high accuracy not only outside of the layers, but also near the layers.

In this paper, we develop a novel Wavelet-based Edge Multiscale Finite Element Method (WEMsFEM) \cite{MR3980476, MR4265650,MR4343247,GL18} which relies an edge multiscale ansatz space that has a good approximation property with very low requirement on the solution, and thereby bypass the requirement on a very fine mesh. This edge multiscale ansatz space is constructed in each local domain by solving a homogeneous convection-dominated diffusion equation with hierarchical bases as Dirichlet boundary data and then use partition of unity functions to glue the local basis functions together as global multiscale basis functions. WEMsFEM shares a certain similarity with the Exponentially Convergent Multiscale Finite Element Method (ExpMsFEM) \cite{MR4269488,ExpMsFEM2023,MR4613768} in that both methods utilize the approximation over the coarse skeleton to define the local multiscale basis functions. The most  striking differences are twofold. Firstly, WEMsFEM can be regarded as a variant of the Partition of Unity Finite Element Methods (PUFEM) \cite{melenk1996partition}, which decomposes the global approximation space into summation of local approximation spaces, while ExpMsFEM uses edge localization and coupling to communicate local and global approximations. Secondly, the local multiscale space in WEMsFEM is established by wavelets or hierarchical bases, and it inherits their intrinsic hierarchical structure. Consequently, there is no need for a further model reduction or deriving important modes by means of Singular Value Decomposition (SVD).
\subsection{Main contributions}
\begin{enumerate}
\item The key idea of WEMsFEM is to transfer the approximation properties over the coarse skeleton to its interior by means of the transposition method \cite{MR0350177}. This is rigorously justified for Problem \eqref{eq:original} in \cref{sec:appendix}.

\item To approximate the solution $u$ to Problem \eqref{eq:original}, we construct an edge multiscale ansatz space with level parameter $\ell$ which encodes key microscale information of the solution such that a guaranteed approximation rate is achieved, which is established in \cref{prop:glo-proj}.
\item The converge rate of WEMsFEMs with respect to the level parameter $\ell$ is proved rigorously without demanding grids, see \cref{prop:wavelet-basedconv}.
\item We conduct extensive numerical experiments both in 2-d and 3-d with large P\'eclet number and even with high oscillation in the diffusion coefficients, and all experiments demonstrate fast convergence rate of WEMsFEMs with respect to the level parameter $\ell$ and high accuracy both inside and outside of the layers.
\end{enumerate}
\subsection{Organization}
We summarize in \cref{sec:prelim} the basic results on singularly perturbed convection-diffusion problem \eqref{eq:original}, and demonstrate the main idea of our approach. Then we present the basic notation and approximation properties of hierarchical bases in \cref{sec:wavelets} for the sake of completeness, which is followed by the construction of multiscale ansatz space $V_{\text{ms},\ell}$. Based upon this, we propose in \cref{sec:edge} the novel Wavelet-based Edge Multiscale Finite Element Methods (WEMsFEMs), which are the key findings of this paper.
Their theoretical and numerical performance are presented in \cref{sec:error} and \cref{sec:numer}. Finally, we conclude the paper with several remarks in \cref{sec:conclusion}.

\section{Preliminaries}\label{sec:prelim}
We recap in this section important properties and numerical challenges of Problem \eqref{eq:original}, and its finite element approximation.
\subsection{Weak formulation and {\it a priori } estimate}
We define the bilinear form $a(\cdot,\cdot)$ on $V\times V$ associated to~\eqref{eq:original} by
\begin{align}\label{eq:bilinear-form}
a(u,v):=\epsilon\int_{D}\nabla u\cdot \nabla v\,\mathrm{d}x
    +\int_{D}(\bm{b}\cdot\nabla u)\,v\,\mathrm{d}x
  \qquad\text{for all }u,v\in V.
\end{align}
Since $\nabla\cdot \bm{b}=0$, an integration by parts implies that
the bilinear form $a(\cdot,\cdot)$ is $V$-elliptic, i.e.,
\begin{equation}\label{eqn:v-elliptic}
\begin{aligned}
  a(v,v)&= \epsilon \normL{\nabla v}{D}^2
   \qquad\text{ for all }v\in V.
\end{aligned}
\end{equation}
Nevertheless, this $V$-ellipticity vanishes when $\epsilon\to 0$, which makes the standard finite element method unstable.

An application of the Poincar\`{e}-Friedrichs inequality leads to the boundedness of the bilinear form,
\begin{align*}
|a(u,v)|\leq \Cb\normL{\nabla u}{D}\normL{\nabla v}{D}.
\end{align*}
Here, the positive constant $\Cb$ may depend on the diameter of the domain $D$ and the $L^{\infty}$-norm of the velocity $\bm{b}$.

The weak formulation for problem \eqref{eq:original} is to find $u\in V$ such that
\begin{align}\label{eqn:weakform}
a(u,v)=(f,v)_{D} \quad \text{for all } v\in V.
\end{align}

Furthermore, this weak solution $u$ is H\"{o}lder continuous,
\begin{proposition}[H\"{o}lder estimate]\label{prop:holder-glo}
\noteLg{Let $d\leq 3$. }For some $\alpha\in (0,1)$, there holds
\begin{align}\label{eq:holder}
\|u\|_{C^{0,\alpha}(\noteLi{\overline{D}})}\leq \mathrm{C}(\alpha)\left(\mathrm{Pe}\normL{\nabla u}{D}+ \epsilon^{-1}\normL{f}{D}\right).
\end{align}
Here, the positive constant $\mathrm{C}(\alpha)$ is independent of $\epsilon$.
\end{proposition}
\begin{proof}
Note that Problem \eqref{eq:original} is equivalent to
\begin{align*}
-\Delta u=\epsilon^{-1}\left(-{\bm{b}}\cdot\nabla u+f\right)& \qquad\text{ in } D,&\\
u=0&\qquad \text{ on } \partial D.&
\end{align*}
On the one hand, the right hand side can be bounded by
\begin{align}\label{eq:rhs-sss}
\normL{-\epsilon^{-1}{\bm{b}}\cdot\nabla u+\epsilon^{-1}f}{D}\leq
\mathrm{Pe}\normL{\nabla u}{D}+\epsilon^{-1}\normL{f}{D}.
\end{align}
On the other hand, the H\"{o}lder estimate for Laplacian \cite[Remark 6.5]{griepentrog2001linear} leads to
\begin{align*}
\|u\|_{C^{0,\alpha}(\noteLi{\overline{D}})}\leq \mathrm{C}(\alpha)\normL{-\epsilon^{-1}{\bm{b}}\cdot\nabla u+\epsilon^{-1}f}{D}
\end{align*}
for some constant $\mathrm{C}(\alpha)$ depends only on the domain $D$ and its boundary $\partial D$. Together with \eqref{eq:rhs-sss}, we obtain the desired assertion.
\end{proof}
\subsection{Finite Element discretization}
To discretize problem \eqref{eq:original},
let $\mathcal{T}^H$ be a regular \noteLi{quasi-uniform} partition of the domain $D$ into \noteLi{quadrilaterals (2d) or cubes (3d)} with a mesh size $H$. For ease of presentation, we denote the mesh P\'{e}clet number $\PeH$ of the coarse mesh $\mc{T}^H$ by
\[
\PeH:=H\|{\bm{b}}\|_{L^{\infty}(D)}/\epsilon.
\]
The vertices of $\mathcal{T}^H$
are denoted by $\{O_i\}_{i=1}^{N}$, with $N$ being the total number of coarse nodes.
The coarse neighborhood associated with the node $O_i$ is denoted by
\begin{equation} \label{neighborhood}
\omega_i:=\text{int} \bigcup\{ \overline{K}: K\in\mathcal{T}^H \text{ satisfying } O_i\in \overline{K}\}.
\end{equation}
The overlapping constant $\Cov$ is defined by
\begin{align}\label{eq:overlap}
\Cov:=\max\limits_{K\in \mathcal{T}^{H}}\#\{O_i: K\subset\omega_i \text{ for } i=1,2,\cdots,N\}.
\end{align}
We assume that $\epsilon\ll 1$ is a small parameter, and that  the coarse mesh $\mathcal{T}^H$ fails to resolve the small parameter, i.e., $\PeH\gg 1$.

Over the coarse mesh $\mathcal{T}^H$, let $V_H$ be the conforming piecewise
\noteLg{bilinear} finite element space,
\[
V_H:=\{v\in \mathcal{C}(D): V|_{T}\in \noteLi{\mathbb{Q}_{1}} \text{ for all } T\in \mathcal{T}^H\},
\]
where $\mathbb{Q}_1$ denotes the space of \noteLi{bilinear} polynomials. Then the classical Galerkin approximation of Problem \eqref{eqn:weakform} reads as finding $u_H\in V_H$, satisfying
\begin{align}\label{eqn:weakform_h}
a(u_H,v_H)=(f,v_H)_{D} \quad \text{ for all } v_H\in V_H.
\end{align}
A straightforward calculation leads to the following estimate
\begin{align*}
 \normL{\nabla(u-u_{H})}{D}
&\lesssim(1+\PeH)\min\limits_{w\in  V_{H}}\normL{\nabla(u-w)}{D}.
 \end{align*}
 \noteLg{Then a standard interpolation property implies}
\begin{align}\label{eq:q1element-error}
 \normL{\nabla(u-u_{H})}{D}&\lesssim(1+\PeH)H\|u\|_{H^2(D)}.
 \end{align}
 Here, and throughout this paper, the notation $A \lesssim B$ means $A\leq c B$ for some benign constant that does not depend on the small parameter $\epsilon$ and the mesh size $H$. On each coarse neighborhood $\omega_i$, the local bilinear form $a_{\omega_i}(\cdot,\cdot)$ is defined analogous to \eqref{eq:bilinear-form} by replacing the global domain $D$ with $\omega_i$. 
 \begin{remark}
If additional stabilization in the spirit of the Streamline Upwind Petrov-Galerkin (SUPG) \cite{MR1365381} is adopted in \eqref{eqn:weakform_h}, then the error estimate \eqref{eq:q1element-error} can be improved by replacing $\PeH$ with $\sqrt{\PeH}$. 
\end{remark}
 Even though the hidden constant in the estimate above is independent of the small parameter $\epsilon$, $\|u\|_{H^2(D)}$ relies on $\epsilon$. 
\subsection{Main objective}
The main goal of this paper is to construct a multiscale ansatz space $V_{\text{ms},\ell}$ with a non-negative level parameter $\ell\in\mathbb{N}$ that incorporates certain {\em a priori} information of the exact solution $u$, such that the following quantity
 \begin{align}\label{eq:approximation-goal}
 \min\limits_{w\in  V_{\text{ms},\ell}}\normL{\nabla(u-w)}{D}
 \end{align}
has a certain decay rate with respect to the number of basis functions in $V_{\text{ms},\ell}$.
\section{Edge multiscale ansatz space $V_{\text{ms},\ell}$}
We introduce in this section the methodology for the construction of $V_{\text{ms},\ell}$ to obtain a guaranteed decay rate in \eqref{eq:approximation-goal} both inside and outside of the layers. 
\subsection{Local-Global splitting}
To start, note that the solution $u$ from \eqref{eq:original} satisfies the following equation
\begin{equation*}
\begin{aligned}
\mathcal{L}_i u\noteLi{:=-\epsilon\Delta u+\bm{b}\cdot\nabla u}&=f \quad&&\text{ in } \omega_i,
\end{aligned}
\end{equation*}
which can be split into the summation of two parts, namely
\begin{align}\label{eq:decomp1}
u|_{\omega_i}=u^{i,\roma}+u^{i,\romb}.
\end{align}
Here, the two components $u^{i,\roma}$ and $u^{i,\romb}$ are respectively given by
\begin{equation}\label{eq:u-roma1conv}
\left\{
\begin{aligned}
\mathcal{L}_i u^{i,\roma}&=f \quad&&\text{ in } \omega_i\\
u^{i,\roma}&=0\quad&&\text{ on }\partial \omega_i,
\end{aligned}
\right.
\end{equation}
and
\begin{equation}\label{eq:111}
\left\{
\begin{aligned}
\mathcal{L}_i u^{i,\romb}&=0 \quad&&\text{ in } \omega_i\\
u^{i,\romb}&=u\quad&&\text{ on }\partial \omega_i.
\end{aligned}
\right.
\end{equation}
This implies that $u^{i,\roma}$ contains the local information of the source term $f$, and $u^{i,\romb}$ encodes the global information over the coarse skeleton $\mathcal{E}^H:=\cup\{\partial K: K\in\mathcal{T}^H\}$.
Note that $u^{i,\roma}$ can be solved in each coarse neighborhood $\omega_i$ for $i=1,2,\cdots,N$ in parallel. \workLi{In practice, it can be solved by the stabilized methods such as the SUPG with sufficient accuracy.}

Next, we introduce a partition of unity $\{\chi_i\}_{i=1}^N$ subordinate to the cover $\{\omega_i\}_{i=1}^N$ to induce a global decomposition from \eqref{eq:decomp1}.
Recall that $O_i$ denotes the ith coarse node in the coarse mesh $\mathcal{T}^H$. We assume this Partition of Unity $\{\chi_i\}_{i=1}^N$ satisfies the following properties:
\begin{equation*}
\left\{
\begin{aligned}
{\text{supp}(\chi_i)}&\subset\omega_i\\
\sum\limits_{i=1}^{N}\chi_i&=1 \text{ in } D\\
 \| \chi_i\|_{L^{\infty}(\omega_i)}&\leq C_{\infty}\\
\|\nabla \chi_i\|_{L^{\infty}(\omega_i)}&\leq C_{\text{G}}H^{-1}
\end{aligned}
\right.
\end{equation*}
for some positive constants $C_{\infty}$ and $ C_{\text{G}}$, which is named the $(\Cov,C_{\infty},C_{\text{G}})$ partition of unity \cite{melenk1996partition}.
For example, the standard bilinear basis functions on the coarse mesh $\mathcal{T}^{H}$ are the partition of unity function with $C_{\infty}=C_{\text{G}}=1$, \noteLi{which are utilised in our numerical tests.} Note that there are many other alternatives for the partition of unity functions besides using this bilinear basis functions. For instance, one can utilize the flat top type of partition of unity functions, cf. \cite{Griebel.Schweitzer:2000}.

Consequently, we can represent $u$ as a summation of local parts,
\begin{align*}
u=\left(\sum_{i=1}^{N}\chi_i\right) u=\sum_{i=1}^{N}\left(\chi_i u|_{\omega_i}\right).
\end{align*}
Inserting the local splitting \eqref{eq:decomp1}, we derive
\begin{align}
u=\sum_{i=1}^{N}\chi_i \left(u^{i,\roma}+u^{i,\romb}\right)
&=\sum_{i=1}^{N}\chi_iu^{i,\roma}+\sum_{i=1}^{N}\chi_iu^{i,\romb}\nonumber\\
&:=u^{\roma}+u^{\romb}.\label{eq:glo-decomp}
\end{align}
Here, the first term
\begin{align}\label{eq:bubble-glo}
u^{\roma}:=\sum_{i=1}^N\chi_i u^{i,\roma}
\end{align}
denotes the global bubble part which contains information of the source term $f$, and the second term
 \begin{align}\label{eq:harmonic-glo}
u^{\romb}:=\sum_{i=1}^N\chi_i u^{i,\romb}.
\end{align}
is the global Harmonic extension part that encodes global information over the coarse skeleton $\mathcal{E}^{H}$.

Since the global bubble part $u^{\roma}$ can be retrieved locally, our main objective \eqref{eq:approximation-goal} is reduced to seeking a good approximative space $V_{\text{ms},\ell}$ for $u^{\romb}$ in the sense that the following quantity has a certain decay rate,
\begin{align}\label{eq:approximation-goal-h}
 \min\limits_{w\in  V_{\text{ms},\ell}}\normL{\nabla(u^{\romb}-w)}{D}.
 \end{align}
To obtain this approximative space $V_{\text{ms},\ell}$, we aim to design global multiscale basis functions taking the same form as \eqref{eq:harmonic-glo}. Since the partition of unity $\{\chi_i\}_{i=1}^N$ is known, then we only need to design local multiscale basis functions to approximate $u^{i,\romb}$ \eqref{eq:111}. Note that $u^{i,\romb}$ is completely determined by the boundary data $u|_{\partial\omega_i}$, hence the question is reduced to seeking a good approximation to $u|_{\partial\omega_i}$. Note also that if we can obtain a proper approximation to $u|_{\partial\omega_i}$, this approximation property can transfer into the local neighborhood $\omega_i$ thanks to the {\em a priori} estimate established in \cref{sec:appendix}.  

In short, we need to approximate $u|_{\partial\omega_i}$, which, however, has restrictive regularity as stated in \cref{prop:holder-glo}. This motivates us to introduce hierarchical bases, since their approximation properties do not rely on the regularity of the target function. 
We present in \cref{sec:wavelets} the basic definition of hierarchical bases with level parameter $\ell\in\mathbb{N}^+$ on a unit interval, and then define the approximative space $V_{\text{ms},\ell}$ in \cref{subsec:edge-ansatz}.
\subsection{Hierarchical subspace splitting over $I=:[0,1]$}\label{sec:wavelets}
In this section, we introduce the hierarchical bases on the unit interval $I:=[0,1]$, which facilitate hierarchically splitting the space $L^2(I)$ \cite{MR1162107}.
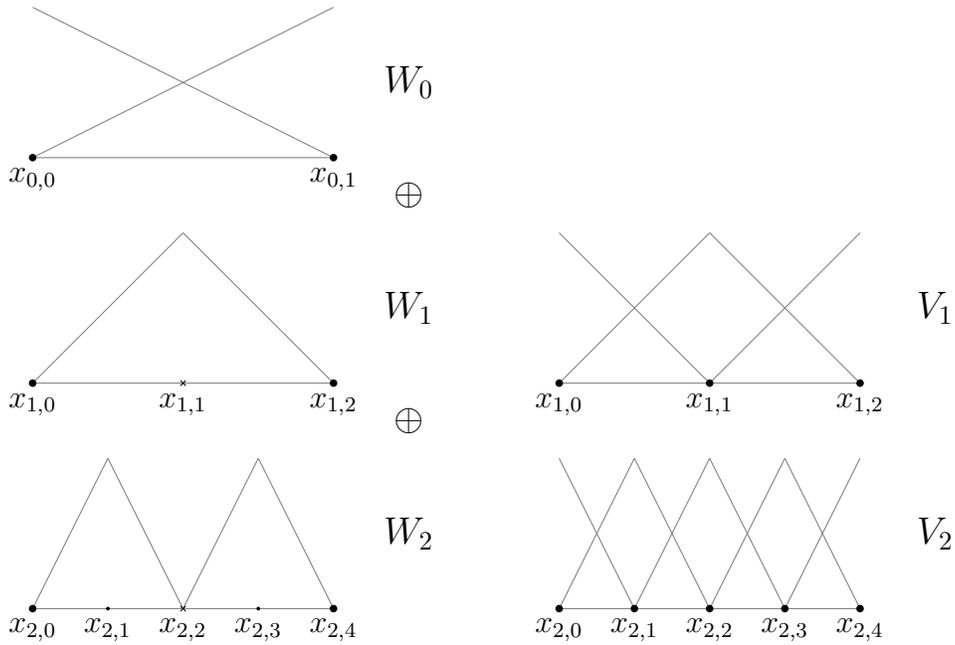
\begin{figure}[htbp]
	\centering
	\def\PointRadius{0.05}
	\tikzset{
		cross/.pic = {
			\draw[rotate = 45] (-#1,0) -- (#1,0);
			\draw[rotate = 45] (0,-#1) -- (0, #1);
		}
	}
	\begin{tikzpicture}[scale=1]
		\draw[gray] (0, 3) -- ++(4, 0);
		\draw[gray] (0, 3) -- ++(4, 2);
		\draw[gray] (4, 3) -- ++(-4, 2);
		\fill[black]  (0, 3) circle (\PointRadius);
		\fill[black]  (4, 3) circle (\PointRadius);
		\node[below] at (0, 3) {\large$ x_{0, 0}$};
		\node[below] at (4, 3) {\large$ x_{0, 1}$};
		
		\draw[gray] (0, 0) -- ++(4, 0);
		\draw[gray] (0, 0) -- ++(2, 2);
		\draw[gray] (4, 0) -- ++(-2, 2);
		\fill[black]  (0, 0) circle (\PointRadius);
		\fill[black]  (4, 0) circle (\PointRadius);
		\path (2,.0) pic[black] {cross=\PointRadius};
		\node[below] at (0, 0) {\large$x_{1, 0}$};
		\node[below] at (2, 0) {\large$x_{1, 1}$};
		\node[below] at (4, 0) {\large$x_{1, 2}$};
		
		\draw[gray] (0, -3) -- ++(4, 0);
		\draw[gray] (0, -3) -- ++(1, 2);
		\draw[gray] (2, -3) -- ++(-1, 2);
		\draw[gray] (2, -3) -- ++(1, 2);
		\draw[gray] (4, -3) -- ++(-1, 2);
		\fill[black]  (0, -3) circle (\PointRadius);
		\fill[black]  (4, -3) circle (\PointRadius);
		\path (2,.-3) pic[black] {cross=\PointRadius};
		\fill[black] (1, -3) circle ({0.5 * \PointRadius});
		\fill[black] (3, -3) circle ({0.5 * \PointRadius});
		\node[below] at (0, -3) {\large$x_{2, 0}$};
		\node[below] at (2, -3) {\large$x_{2, 2}$};
		\node[below] at (4, -3) {\large$x_{2, 4}$};
		\node[below] at (1, -3) {\large$x_{2, 1}$};
		\node[below] at (3, -3) {\large$x_{2, 3}$};
		
		\node at (5, 3+1) {\Large$W_0$};
		\node at (5,  0+1) {\Large$W_1$};
		\node at (5, -3+1) {\Large$W_2$};
		\node at (5, 1.5+1) {\Large$\oplus$};
		\node at (5, -1.5+1) {\Large$\oplus$};

		\draw[gray] (7, 0) -- ++(4, 0);
		\draw[gray] (7, 0) -- ++(2, 2);
		\draw[gray] (9, 0) -- ++(-2, 2);
		\draw[gray] (9, 2) -- ++(2, -2);
		\draw[gray] (7+2, 0) -- ++(2, 2);		
		\fill[black]  (5+2, 0) circle (\PointRadius);
		\fill[black]  (7+2, 0) circle (\PointRadius);
		\fill[black]  (9+2, 0) circle (\PointRadius);
		\node[below] at (5+2, 0) {\large$x_{1, 0}$};
		\node[below] at (7+2, 0) {\large$x_{1, 1}$};
		\node[below] at (9+2, 0) {\large$x_{1, 2}$};
		
		\draw[gray] (0+7, -3) -- ++(4, 0);
		\draw[gray] (0+7, -3) -- ++(1, 2);
		\draw[gray] (8, -3) -- ++(-1, 2);
		\draw[gray] (10, -3) -- ++(-1, 2);
		\draw[gray] (8, -3) -- ++(1, 2);
		\draw[gray] (10, -3) -- ++(1, 2);
		\draw[gray] (2+7, -3) -- ++(-1, 2);
		\draw[gray] (2+7, -3) -- ++(1, 2);
		\draw[gray] (4+7, -3) -- ++(-1, 2);
		\fill[black]  (0+5+2, -3) circle (\PointRadius);
		\fill[black]  (2+5+2, -3) circle (\PointRadius);
		\fill[black]  (6+2, -3) circle (\PointRadius);	
		\fill[black] (10, -3) circle ({1 * \PointRadius});
		\fill[black] (11, -3) circle ({1 * \PointRadius});
		\node[below] at (7, -3) {\large$x_{2, 0}$};
		\node[below] at (1+5+2, -3) {\large$x_{2, 1}$};
		\node[below] at (2+5+2, -3) {\large$x_{2, 2}$};
		\node[below] at (10, -3) {\large$x_{2, 3}$};
		\node[below] at (11, -3) {\large$x_{2, 4}$};

		\node at (12,  0+1) {\Large$V_1$};
		\node at (12, -3+1) {\Large$V_2$};
		
	\end{tikzpicture}
\caption{An illustration of hierarchical bases for $\ell=0,1,2$.}
\end{figure}

Let the level parameter and the mesh size be $\ell$ and $h_{\ell}:=2^{-\ell}$ with $\ell\in \mathbb{N}$, respectively. Then the grid points on level $\ell$ are
\[
x_{\ell,j}=j\times h_{\ell},\qquad 0\leq j\leq 2^{\ell}.
\]

We can define the basis functions on level $\ell$ by
\begin{equation*}
\psi_{\ell,j}(x)=
\left\{
\begin{aligned}
&1-|x/h_{\ell}-j|, &&\text{ if }  x\in [(j-1)h_{\ell},(j+1)h_{\ell}]\cap [0,1],\\
&0, &&\text{ otherwise.}
\end{aligned}
\right.
\end{equation*}
Define the set on each level $\ell$ by
\begin{equation*}
B_{\ell}:=\Bigg\{
j\in\mathbb{N}\Bigg|
\begin{aligned}
&j=1,\cdots,2^{\ell}-1, j \,\rm{ is\, odd }, &&\rm{if }\,\ell>0\\
&j=0,1,&&\rm{ if }\,\ell=0
\end{aligned}
\Bigg\}.
\end{equation*}
The subspace of level $\ell$ is
\[
W_{\ell}:=\text{span}\{\psi_{\ell,j}:\quad j\in B_{\ell}\}.
\]
We denote $V_{\ell}$ as the subspace in $L^2(I)$ up to level $\ell$, which is defined by the direct sum of subspaces
\[
V_{\ell}:=\oplus_{m\leq\ell}W_{m}.
\]
Consequently, this yields the hierarchical structure of the subspace $V_{\ell}$, namely,
\[
V_{0}\subset V_{1}\subset \cdots\subset V_{\ell}\subset V_{\ell+1}\cdots
\]
Furthermore, the following hierarchical decomposition of the space $L^2(I)$ holds
\[
L^2(I)=\lim_{\ell\to\infty}\oplus_{m\leq\ell}W_{m}.
\]


Note that one can derive the hierarchical decomposition of the space \workLg{$L^2(I^{d-1})$ for $d=2,3$ by means of tensor product, which is denoted as $V_{\ell}^{\otimes^{d-1}}$. Note further that we will use the subspace $V_{\ell}^{\otimes^{d-1}}$ to approximate the restriction of the exact solution $u$ on the coarse skeleton $\partial\mathcal{T}^H:=\cup_{T\in \mathcal{T}^H}\partial T$.}

Next, we present the approximation properties of the hierarchical space $V_{\ell}$, which can be derived from standard $L^2$-projection error \cite{MR520174} combining with the interpolation method.


\begin{proposition}[Approximation properties of the hierarchical space $V_{\ell}$]\label{prop:approx-wavelets}
Let $d=2,3$, $s>0$ and let $\mathcal{I}_{\ell}: L^2(I^{d-1})\to V_{\ell}^{\otimes^{d-1}}$ be $L^2$-projection for each level $\ell\geq 0$, then there holds,
\begin{align}\label{eq:approx-wavelets}
\|v-\mathcal{I}_{\ell}v\|_{L^2(I^{d-1})}&\lesssim 2^{-s\ell}|v|_{H^s(I^{d-1})}\quad\text{for all }v\in H^s(I^{d-1}).
\end{align}
Here, $|\cdot|_{H^s(I^{d-1})}$ denotes the Gagliardo seminorm in the fractional Sobolev space (or Slobodeskii space) $H^s(I^{d-1})$, given by
\begin{align*}
|v|^2_{H^s(I^{d-1})}:=\int_{I^{d-1}}\int_{I^{d-1}}\frac{|v(x)-v(y)|^2}{|x-y|^{d-1+2s}}\mathrm{d}x\mathrm{d}y.
\end{align*}
The corresponding full norm is denoted as $\|\cdot\|_{H^s(I^{d-1})}$.
\end{proposition}
\subsection{Edge multiscale ansatz space}\label{subsec:edge-ansatz}
Now we are ready to define the approximative space $V_{\text{ms},\ell}$. \workLg{We will illustrate it on the case with $d=2$ for the sake of simplicity, even though our methodology can be extended to $d=3$ naturally.}

\begin{figure}[htbp]
	\centering
	\begin{tikzpicture}[scale=1]
		\draw[step=1.0, black, very thick] (-0, -0) grid (4, 4);
		\foreach \x in {0,...,4}
		\foreach \y in {0,...,4}{
			\fill (1.0 * \x, 1.0 * \y) circle (1.5pt);
		}
		\fill [red] (2.0 , 2.0 ) circle (1.5pt);
		\fill[green, opacity=0.4] (1.0, 1.0) rectangle (3.0, 3.0);
		\node at (1.8, 1.8) {$\omega_i$};
		
		\draw[step=1.0, black, very thick] (5,4) grid (7, 6);
		\foreach \x in {5,7}
		\foreach \y in {4,6}{
			\fill [blue] (1.0 * \x, 1.0 * \y) circle (1.5pt);
		}
		
		\draw[step=1.0, black, very thick] (5,1) grid (7, 3);
		
		\foreach \x in {5,6,7}
		\foreach \y in {1,2,3}{
			\fill [blue] (1.0 * \x, 1.0 * \y) circle (1.5pt);
		}

		\draw[step=1.0, black, very thick] (5,-2) grid (7, 0);
		
		\foreach \x in {5,6,7}
		\foreach \y in {-2,-1,0}{
			\fill [blue] (1.0 * \x, 1.0 * \y) circle (1.5pt);	
		}
		
		\foreach \x in {5.5,6.5}
		\foreach \y in {-2,0}{
			\fill [blue] (1.0 * \x, 1.0 * \y) circle (1.5pt);		
		}
		
		\foreach \x in {5,7}
		\foreach \y in {-1.5,-0.5}{
			\fill [blue] (1.0 * \x, 1.0 * \y) circle (1.5pt);
		}
		
		\fill [red] (6,-1) circle (1.5pt);
		\fill [red] (6,2) circle (1.5pt);
		\fill [red] (6,5) circle (1.5pt);
		
		\draw [-to](5,4.6) -- (3.0,2.6);
		\draw [-to](5,2.1) -- (3.0,2.1);
		\draw [-to](5,-0.4) -- (3.0,1.4);
		
		\node at (7.8, -1) {$\ell=2$};
		\node at (7.8, 2) {$\ell=1$};
		\node at (7.8, 5) {$\ell=0$};
	\end{tikzpicture}
	\caption{Grid points for $\ell=0,1,2$ over $\partial\omega_i$.}
\label{fig:grid-points-local}
\end{figure}
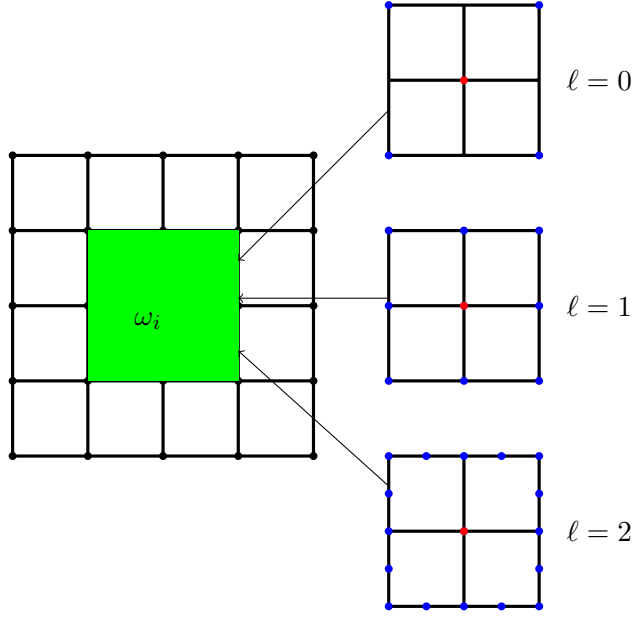
To begin, we first define the linear space over the boundary of each coarse neighborhood $\partial\omega_i$. Let the level parameter $\ell\in \mathbb{N}$ be fixed, and let $\Gamma_{i}^k$ with $k=1,2,3,4$ be a partition of $\partial\omega_i$ with no mutual intersection, i.e., $\cup_{k=1}^{4}\overline{\Gamma_{i}^k}=\partial\omega_i$ and $\Gamma_{i}^k\cap \Gamma_{i}^{k'}=\emptyset$ if $k\neq k'$. Furthermore,
we denote $V_{i,\ell}^k\subset \noteLi{C(\partial\omega_{i})}$ as the linear space spanned by hierarchical bases up to level $\ell$ on each coarse edge $\Gamma_{i}^{k}$ \noteLi{and continuous over $\partial\omega_i$}, then the local edge space $V_{i,\ell}$ defined over $\partial\omega_i$ is the smallest linear space having $V_{i,\ell}^k$ as a subspace. Let $\{\psi_{i,\ell}^j\}_{j=1}^{2^{\ell+2}}$ and 
$\{x_{i,\ell}^j\}_{j=1}^{2^{\ell+2}}$ be the nodal basis functions and the associated nodal points for $V_{i,\ell}$, then we can represent the local edge space $V_{i,\ell}$ by  
\begin{align}\label{eq:local-edge}
V_{i,\ell} := \text{span} \left\{\psi_{i,\ell}^j : 1 \leq j \leq  2^{\ell+2}\right\}.
\end{align}
We depict the grid points $\{x_{i,\ell}^j\}_{j=1}^{2^{\ell+2}}$ of the hierarchical bases over $\partial\omega_i$ in Figure \ref{fig:grid-points-local} for level parameter $\ell$ ranging from 0 to 2. 

Next, we introduce the local multiscale space over each coarse neighborhood $\omega_i$, which is defined by 
\begin{align}\label{eq:local-multiscale}
\mathcal{L}^{-1}_i(V_{i,\ell})
:= \text{span} \left\{\mathcal{L}^{-1}_i(\psi_{i,\ell}^j) : 1 \leq j \leq  2^{\ell+2}\right\}.
\end{align}
Here, $\mathcal{L}^{-1}_i (\psi_{i,\ell}^j):=v\in H^1(\omega_i)$ \noteLi{is the solution to the following local problem}, 
\begin{equation}
\label{eq:Li}
  \left\{ \begin{aligned}
          \mathcal{L}_i v&:=-\epsilon\Delta v+\bm{b}\cdot \nabla v=0\quad \mbox{in }\omega_i,\\
          v&=\psi_{i,\ell}^j\quad \mbox{on }\partial\omega_i.
  \end{aligned}\right.
\end{equation}
By this construction, the dimension of $\mathcal{L}^{-1}_i (V_{i,\ell})$ equals $2^{\ell+2}$. In practice, a number of $2^{\ell+2}$ local problems \eqref{eq:Li} is solved by the stabilized methods such as the SUPG with sufficient accuracy to obtain the local multiscale space $\mathcal{L}^{-1}_i(V_{i,\ell})$, which can be solved in parallel and thus has a much lower computational complexity than the original problem \eqref{eq:original}. 

Finally, the edge multiscale ansatz space is defined by the Partition of Unity $\{\chi_i\}_{i=1}^N$,
\begin{align}\label{eq:global-multiscale}
V_{\text{ms},\ell} := \text{span} \left\{\chi_i\mathcal{L}^{-1}_i(\psi_{i,\ell}^j) : \,  \, 1 \leq i \leq N \text{ and } 1 \leq j \leq  2^{\ell+2}\right\}.
\end{align}
\section{Wavelet-based Edge Multiscale Finte Element Method}\label{sec:edge}
We present in this section the Wavelet-based Edge Multiscale Finte Element Method (WEMsFEM) to solve \eqref{eq:original} based upon the global splitting \eqref{eq:glo-decomp}. 


To get a good approximation to $u^{\romb}$, we solve for $u_{\text{ms},\ell}^{\romb}\in V_{\text{ms},\ell}$, s.t.,
\begin{align}\label{eqn:weakform_h-romb}
a\left(u_{\text{ms},\ell}^{\romb},v_{\text{ms},\ell}\right)=(f,v_{\text{ms},\ell})_{D}- a( u^{\roma},v_{\text{ms},\ell})
 \quad \text{ for all } v_{\text{ms},\ell}\in V_{\text{ms},\ell}.
\end{align}
Then the multiscale solution $u_{\text{ms},\ell}$ is defined as the summation of these two parts,
\begin{align}\label{eq:u-ms}
u_{\text{ms},\ell}:=u^{\roma}+u_{\text{ms},\ell}^{\romb}.
\end{align}
Note that the proposed edge multiscale ansatz space $V_{\text{ms},\ell}$ is independent of the source term $f$, and hence one can approximate $u$ by \eqref{eq:u-ms} using the same edge multiscale ansatz space $V_{\text{ms},\ell}$ for different source term $f$, which is encountered frequently in many applications such as the uncertainty quantification.

Our main algorithm is summarized in \cref{algorithm:wavelet}.

\begin{remark}[Computational complexity for \cref{algorithm:wavelet}]
\cref{algorithm:wavelet} involves the computation of the global bubble part $u^{\roma}$ and the edge multiscale ansatz space $V_{\text{ms},\ell}$. The former requires solving the local problem \eqref{eq:u-roma1conv} and the latter requires solving the local problem \eqref{eq:Li}. In total, there is $\mathcal{O}(H^{-d}(1+2^{\ell+d}))$ number of local solvers needed for a d-dimension problem, which can be performed in parallel. Note that the dimension of the edge multiscale ansatz space $V_{\text{ms},\ell}$ grows exponentially as the level parameter $\ell$ increases, which implies that the computational complexity to solve \eqref{eqn:weakform_h-romb} grows exponentially with respect to $\ell$. Nevertheless, we observe from the numerical experiments that $\ell=2$ is sufficient for all cases. Moreover, the computational degree of freedom is restricted to the coarse skeleton $\partial\mathcal{T}^H$, which is much smaller than the degree of freedom in $V_h$. 
\end{remark}

\begin{remark}[Local decomposition \eqref{eq:decomp1}]
The local decomposition \eqref{eq:decomp1} is different from the one constructed for the elliptic problems with heterogeneous coefficients \cite{MR3980476} in that we have to solve $u^{i,\roma}$ in our current work. In the previous work \cite{MR3980476}, $u^{\roma}$ is thrown away since its energy norm is $\mathcal{O}(H)$. However, $u^{\roma}$ is not negligible in this work because of the estimate \eqref{eq:local-part1} to be presented in Section \ref{sec:error}, arising from the presence of small parameter $\epsilon$. 
\end{remark}

\begin{algorithm}[htbp]
\caption{Wavelet-based Edge Multiscale Finte Element Method (WEMsFEM)}
\label{algorithm:wavelet}
\begin{algorithmic}[1]
    \REQUIRE The level parameter $\ell\in \mathbb{N}$; coarse neighborhood $\omega_i$ and its four coarse edges $\Gamma_{i}^{k}$ with
    $k=1,2,3,4$;
    the subspace $V_{i,\ell}^k\subset L^2(\Gamma_{i}^{k})$ up to level $\ell$ on each coarse edge $\Gamma_{i}^{k}$.
\ENSURE $u_{\text{ms},\ell}$
\STATE Construct the local edge space $V_{i,\ell}$ \eqref{eq:local-edge}   
\STATE Calculate the local multiscale space $\mathcal{L}^{-1}_i (V_{i,\ell})$ \eqref{eq:local-multiscale}
\STATE Construct the global multiscale space $V_{\text{ms},\ell}$ \eqref{eq:global-multiscale}
 \STATE Solve for the local bubble function $u^{i,\roma}$ from \eqref{eq:u-roma1conv} for $i=1,\cdots,N$
\STATE Calculate the global bubble function $u^{\roma}$ from \eqref{eq:bubble-glo}
\STATE Solve for $u_{\text{ms},\ell}^{\romb}\in V_{\text{ms},\ell}$ from \noteLi{\eqref{eqn:weakform_h-romb}}
\STATE Obtain $u_{\text{ms},\ell}$ from \eqref{eq:u-ms}.
\end{algorithmic}
\end{algorithm}
To analyze the convergence of Algorithm \ref{algorithm:wavelet}, we next introduce the global projection operator $\mathcal{P}_{\ell}$ of level $\ell$: $V\to V_{\text{ms},\ell}$. Since the edge multiscale ansatz space $V_{\text{ms},\ell}$ is generated by the local multiscale space $\mathcal{L}^{-1}(V_{i,\ell})$ by means of the partition of unity \eqref{eq:global-multiscale}, we only need to define local projection operator $\mathcal{P}_{i,\ell}$ of level $\ell$: $L^2(\partial\omega_i)\to \mathcal{L}^{-1}(V_{i,\ell})$ by 
\begin{align}\label{eq:projectionEDGE}
\mathcal{P}_{i,\ell}v:=
\mathcal{L}_{i}^{-1}(\mathcal{I}_{i,\ell}v).
\end{align}
Here, $\mathcal{I}_{i,\ell}:L^2(\partial\omega_i)\to V_{i,\ell}$ denotes the $L^2$-projection.

Note that $\mathcal{P}_{i,\ell}v|_{\partial\omega_i}$ is $L^2$-projection onto $V_{i,\ell}$, this is because 
\begin{align*}
\mathcal{P}_{i,\ell}v|_{\partial\omega_i}
&=\mathcal{I}_{i,\ell}v.
\end{align*}
Note that any $v\in V$ can be expressed by
\begin{align*}
v=\sum_{i=1}^N \chi_i v|_{\omega_i}, 
\end{align*}
then the global interpolation $\mathcal{P}_{\ell}$ of level $\ell$: $ V\to V_{\text{ms},\ell}$ can be defined by means of the local projection \eqref{eq:projectionEDGE},
\begin{align}\label{eq:glo-proj}
\mathcal{P}_{\ell}(v):=\sum_{i=1}^N\chi_i (\mathcal{P}_{i,\ell}(v|_{\partial\omega_i})).
\end{align}
\section{Error estimate}\label{sec:error}
We present in this section the convergence analysis for Algorithm \ref{algorithm:wavelet}, which is divided into two steps. First, we derive the approximation properties of the edge multiscale ansatz space $V_{\text{ms},\ell}$ defined in \eqref{eq:global-multiscale}, which is based upon establishing the local approximation properties of the local multiscale space $\mathcal{L}^{-1}_i(V_{i,\ell})$ in each coarse neighborhood $\omega_i$. Second, we provide a quasi-optimal result by C\'{e}a type of estimate in Theorem \ref{prop:wavelet-basedconv}.

Recall that $\{\chi_i\}_{i=1}^N$ is a $(\Cov,C_{\infty},C_{\text{G}})$ partition of unity subordinate to the cover $\{\omega_i\}_{i=1}^N$, then a slight modification of {\cite[Theorem 2.1]{melenk1996partition}} leads to the following result,
\begin{theorem}[Local-global approximation]
\label{thm:pum}
Assume that $D\subset\mathbb{R}^d$ for $d=2,3$. Let $v_i\in H^1(\omega_i)$ for $i=1,\cdots,N$, and let
\begin{align*}
v=\sum_{i=1}^N \chi_i v_i\in H^1(D),
\end{align*}
then there holds
\begin{align}
\|v\|_{L^2(D)}&\leq \sqrt{\Cov}C_{\infty}\left(\sum_{i=1}^N\|v_i\|_{L^2(\omega_i)}^2\right)^{1/2},\nonumber\\
\|\nabla v\|_{L^2(D)}&\leq \sqrt{2\Cov}\left(\sum_{i=1}^NC_{\text{G}}^2H^{-2}\|v_i\|_{L^2(\omega_i)}^2
+\|\chi_i\nabla v_i\|_{L^2(\omega_i)}^2\right)^{1/2}.\label{eq:glo-err-pum}
\end{align}
\end{theorem}
\begin{remark}
The global estimate in \eqref{eq:glo-err-pum} is slightly different from \cite[Theorem 2.1]{melenk1996partition} since the local estimate $\|\chi_i\nabla v_i\|_{L^2(\omega_i)}$ is utilized instead of $\|\nabla v_i\|_{L^2(\omega_i)}$. The main reason is that our proposed local multiscale space $\mathcal{L}_i^{-1}(V_{i,\ell})$ do not have approximation rate in the sense that the following quantity
\begin{align*}
 \min\limits_{w\in \mathcal{L}_i^{-1}(V_{i,\ell})}
\normL{\nabla(u^{i,\romb}-w)}{\omega_i}
 \end{align*}
may not admit a certain decay rate with respect to the level parameter $\ell$. However, we can prove the decay rate of the following modified quantity,
\begin{align*}
 \min\limits_{w\in \mathcal{L}_i^{-1}(V_{i,\ell})}
\normL{\chi_i\nabla(u^{i,\romb}-w)}{\omega_i}.
 \end{align*}
\end{remark}
%
\begin{lemma}[{\em A priori} estimate for $u^{\romb}$]\label{lem:parta}
Assume that $D\subset\mathbb{R}^d$ for $d=2,3$. Let $u^{i,\romb}$ and $u^{\romb}$ be defined in \eqref{eq:111} and \eqref{eq:glo-decomp} and let $f\in L^2(D)$. Then it holds
\begin{align*}
\|\nabla u^{i,\romb}\|_{L^2(D)}&\lesssim \noteLi{\frac{H}{\epsilon}} \normL{f}{\omega_i}+\|\nabla u\|_{L^2(\omega_i)}\\
\|\nabla u^{\romb}\|_{L^2(D)}&\lesssim \noteLi{\frac{H}{\epsilon}} \normL{f}{D}+\|\nabla u\|_{L^2(D)}.
\end{align*}
\end{lemma}
\begin{proof}
On the one hand,
multiplying \eqref{eq:u-roma1conv} by $u^{i,\roma}$ and then taking its integral over $\omega_i$ lead to
\begin{align*}
\epsilon\normL{\nabla u^{i,\roma}}{\omega_i}^2
=\int_{\omega_i}f u^{i,\roma}\dx.
\end{align*}
Notice that $u^{i,\roma}=0$ on $\partial\omega_i$, therefore, an application of the Poincar\'{e} inequality reveals
\begin{align*}
\normL{u^{i,\roma}}{\omega_i}+H\normL{\nabla u^{i,\roma}}{\omega_i}\lesssim \noteLi{\frac{H^2}{\epsilon}} \normL{f}{\omega_i}.
\end{align*}
Together with the local splitting \eqref{eq:decomp1}, this implies 
\begin{align*}
\normL{\nabla u^{i,\romb}}{\omega_i}&=\normL{\nabla (u-u^{i,\roma})}{\omega_i}\leq \normL{\nabla u^{i,\roma}}{\omega_i}+\normL{\nabla u}{\omega_i}\\
&\lesssim \noteLi{\frac{H}{\epsilon}} \normL{f}{\omega_i}+\|\nabla u\|_{L^2(\omega_i)}.
\end{align*}
On the other hand, using Theorem \ref{thm:pum}, we obtain
\begin{align}\label{eq:local-part1}
\normL{u^{\roma}}{D}+H\normL{\nabla u^{\roma}}{D}\lesssim \noteLi{\frac{H^2}{\epsilon}}\normL{f}{D}.
\end{align}
Then the global splitting \eqref{eq:glo-decomp} indicates
\begin{align*}
\normL{\nabla u^{\romb}}{D}=\normL{\nabla (u-u^{\roma})}{D}\leq \normL{\nabla u}{D}+\normL{\nabla u^{\roma}}{D},
\end{align*}
which, together with \eqref{eq:local-part1}, leads to the desired assertion.
\end{proof}
Next, we present the approximation properties of the global interpolation operator $\mathcal{P}_{\ell}$ \eqref{eq:glo-proj},
\begin{lemma}[Approximation properties of $\mathcal{P}_{\ell}$]\label{prop:glo-proj}
Assume that $D\subset\mathbb{R}^d$ for $d=2,3$, $s>0$, $f\in L^{2}(D)$ and the level parameter $\ell\in \mathbb{N}$ is non-negative. Let $u\in V$ be the solution to Problem \eqref{eq:original} and the global Harmonic extension $u^{\romb}$ be defined in \eqref{eq:harmonic-glo}. Then there holds
\begin{align}
\normL{u^{\romb}-\mathcal{P}_{\ell}u^{\romb}}{D}
&\lesssim (\PeH^{-1/2}+\PeH^{1/2}) H^{s+1/2}2^{-s\ell}\|u\|_{H^{s+1/2}(D)}, \label{eq:glo-l2}\\
\normL{\nabla(u^{\romb}-\mathcal{P}_{\ell}u^{\romb})}{D}
&\lesssim (\PeH^{-1/2}+\PeH) H^{{s-1/2}}2^{-s\ell}\|u\|_{H^{s+1/2}(D)}.\label{eq:glo-energy}
\end{align}
\end{lemma}
\begin{proof}
Let $\locv{e}{}{}:=u^{\romb}-\mathcal{P}_{\ell}u^{\romb}$ be the global error, then the property of the partition of unity of $\{\chi_i\}_{i=1}^{N}$, together with \eqref{eq:harmonic-glo}, leads to
\[
\locv{e}{}{}=\sum\limits_{i=1}^{N}\chi_i\locv{e}{i}{} \qquad\text{ with }
\qquad\locv{e}{i}{}:=u^{i,\romb}-\mc{P}_{i,\ell}u^{i,\romb}.
\]
Our proof is composed of three steps.

\noindent Step 1. Estimate the local error $\locv{e}{i}{}$ over each local boundary $\partial\omega_i$.  Using the $L^2$-projection error for hierarchical bases \eqref{prop:approx-wavelets} and a scaling argument, we obtain
\begin{align}
\|\locv{e}{i}{}\|_{L^2(\partial \omega_i)}
&=\|u^{i,\romb}-\mc{P}_{i,\ell}u^{i,\romb}\|_{L^2(\partial\omega_i)}\nonumber\\
&\workLg{\lesssim 2^{-\textcolor{black}{s}\ell} {H^{\textcolor{black}{s}}}} |u^{i,\romb}|_{\textcolor{black}{H^{s}(\partial\omega_i)}}\nonumber.
\end{align}
The definition of the local solution $u^{i,\romb}$ \eqref{eq:111}, \textcolor{black}{together with the Sobolev embedding theorem}, implies,
\begin{align}
\|\locv{e}{i}{}\|_{L^2(\partial \omega_i)}
&\workLg{\lesssim 2^{-s\ell} H^{s}}\|u\|_{\textcolor{black}{H^{s+1/2}(\omega_i)}}.\label{eq:yyyy}
\end{align}
\noindent Step 2. Estimate the local error $\locv{e}{i}{}$ over each coarse neighborhood $\omega_i$ mainly by the transposition method established in \cref{sec:appendix}.

Note that each local error $\locv{e}{i}{}$ satisfies the following equation,
\begin{equation*}
\begin{aligned}
\mathcal{L}_i \locv{e}{i}{}:=-\epsilon\Delta \locv{e}{i}{}+\bm{b}\cdot\nabla \locv{e}{i}{}&=0 && \text{ in } \omega_i\\
v&=0 &&\text{ on }\partial \omega_i\cap\partial D.
\end{aligned}
\end{equation*}
\cref{lem:very-weak} and \cref{corollary:very-weak} indicate that the local error $\locv{e}{i}{}$ inside of $\omega_i$ can be bounded by its boundary data,
\begin{align*}
\|\locv{e}{i}{}\|_{L^2(\omega_i)}&\leq {\rm C}_{{\rm weak}}\textcolor{black}{(\PeH^{-1/2}+\PeH^{1/2})H^{1/2}}\|\locv{e}{i}{}\|_{L^2(\partial \omega_i)}\\
\normL{\chi_i\nabla \locv{e}{i}{}}{\omega_i}&\leq {\rm C}_{{\rm weak}}(\textcolor{black}{\PeH^{-1/2}+\PeH})H^{-1/2}\|\locv{e}{i}{}\|_{L^2(\partial \omega_i)}.
\end{align*}
Then together with \eqref{eq:yyyy}, this leads to
\begin{align}
\|\locv{e}{i}{}\|_{L^2(\omega_i)}&\leq {\rm C}_{{\rm weak}}\textcolor{black}{(\PeH^{-1/2}+\PeH^{1/2})}\workLg{ H^{s+1/2}2^{-s\ell} }\|u\|_{\textcolor{black}{H^{s+1/2}(\omega_i)}}\label{eq:11111}\\
\normL{\chi_i\nabla \locv{e}{i}{}}{\omega_i}&\leq {\rm C}_{{\rm weak}} \left(\textcolor{black}{\PeH^{-1/2}+{\PeH}}\right) H^{s-\frac{1}{2}}2^{-s\ell}\|u\|_{\textcolor{black}{H^{s+1/2}(\omega_i)}}.\label{eq:22222}
\end{align}
\noindent Step 3. Estimate the global error by summation of local error. Using \cref{thm:pum}, we obtain
\begin{align*}
\normL{\locv{e}{}{}}{D}
\leq \sqrt{\Cov}C_{\infty}\left(\sum_{i=1}^N\|\locv{e}{i}{}\|_{L^2(\omega_i)}^2\right)^{1/2},
\end{align*}
which, together with \eqref{eq:11111} and the overlapping condition \eqref{eq:overlap} leads to
\begin{align*}
\normL{\locv{e}{}{}}{D}
\leq C_{\text{weak}}\sqrt{\Cov}C_{\infty}\textcolor{black}{(\PeH^{-1/2}+\PeH^{1/2})}{ H^{s+1/2}2^{-s\ell} }
\left(\sum_{i=1}^N\|u\|_{\textcolor{black}{H^{s+1/2}(\omega_i)}}^2\right)^{1/2}.
\end{align*}
Consequently, we derive
\begin{align*}
\normL{\locv{e}{}{}}{D}
\leq C_{\text{weak}}\sqrt{\Cov}C_{\infty}\textcolor{black}{(\PeH^{-1/2}+\PeH^{1/2})}\workLg{H^{s+1/2}2^{-s\ell}}\|u\|_{\textcolor{black}{H^{s+1/2}(D)}}.
\end{align*}
This proves \eqref{eq:glo-l2}.

Next, we prove \eqref{eq:glo-energy}.
Using again \cref{thm:pum}, we obtain
\begin{align}\label{eq:33333}
\|\nabla \locv{e}{}{}\|_{L^2(D)}&\leq \sqrt{2\Cov}\left(\sum_{i=1}^NC_{\text{G}}^2H^{-2}\|\locv{e}{i}{}\|_{L^2(\omega_i)}^2
+\|\chi_i\nabla \locv{e}{i}{}\|_{L^2(\omega_i)}^2\right)^{1/2}.
\end{align}
Combining with \eqref{eq:11111} and \eqref{eq:22222}, we have proved \eqref{eq:glo-energy}, and this completes our proof.
\end{proof}
Finally, we investigate the error between the exact solution $u$ to Problems \eqref{eq:original} and the multiscale solution $u_{{\rm ms},\ell}$, which is derived from first establishment of the quasi-optimality of $u_{{\rm ms},\ell}$ that is analogous to the classical C\'{e}a's lemma, and then apply the approximation properties of the edge multiscale ansatz space stated in \cref{prop:glo-proj}.
\begin{theorem}[Error estimate for \cref{algorithm:wavelet}]\label{prop:wavelet-basedconv}
For $d=2,3$ and  $s>0$. Let $\ell\in \mathbb{N}$ be non-negative. Let $u\in V$ be the solution to Problem \eqref{eq:original} and let $u_{\text{ms},\ell}$ be the multiscale solution defined in \eqref{eq:u-ms}. Then there holds
\begin{equation}\label{eq:waveletErrconv}
\begin{aligned}
\normL{\nabla(u-u_{{\rm ms},\ell})}{D}\lesssim
\left(\PeH^{-1/2}+\textcolor{black}{\PeH^{3/2}}\right) H^{s-1/2}{2^{-s\ell}}\|u\|_{\textcolor{black}{H^{s+1/2}(D)}}.
\end{aligned}
\end{equation}
\end{theorem}
\begin{proof}
We obtain from the global decomposition \eqref{eq:glo-decomp} and \eqref{eq:u-ms},
\begin{align*}
\epsilon\normL{\nabla(u-u_{{\rm ms},\ell})}{D}^2
&=a(u^{\romb}-u_{{\rm ms},\ell}^{\romb},u^{\romb}-u_{{\rm ms},\ell}^{\romb}).
\end{align*}
Since $V_{\text{ms},\ell}\subset V$, \noteLg{together with the definition of $u_{{\rm ms},\ell}^{\romb}$ in \eqref{eqn:weakform_h-romb},} this implies
\begin{align*}
\epsilon\normL{\nabla(u-u_{{\rm ms},\ell})}{D}^2
&=a(u^{\romb}-u_{{\rm ms},\ell}^{\romb},u^{\romb}-\mathcal{P}_{\ell}u^{\romb}).
\end{align*}
The definition of the bilinear from $a(\cdot,\cdot)$ \eqref{eq:bilinear-form}, together with the Cauchy-Schwarz inequality, further leads to
\begin{align*}
&\qquad\epsilon\normL{\nabla(u^{\romb}-u^{\romb}_{{\rm ms},\ell})}{D}^2\\
&=\epsilon\int_D\nabla(u^{\romb}-u_{{\rm ms},\ell}^{\romb})\cdot\nabla(u^{\romb}-\mathcal{P}_{\ell}u^{\romb})\mathrm{d}x
+\int_{D}\bm{b}\noteLg{\cdot} \nabla(u^{\romb}-u_{{\rm ms},\ell}^{\romb})(u^{\romb}-\mathcal{P}_{\ell}u^{\romb})\mathrm{d}x\\
&\leq {\epsilon}\normL{\nabla(u^{\romb}-u_{{\rm ms},\ell}^{\romb})}{D}\normL{\nabla(u^{\romb}-\mathcal{P}_{\ell}u^{\romb})}{D}\\
&+
\|\bm{b}\|_{L^{\infty}(D)}\normL{\nabla(u^{\romb}-u_{{\rm ms},\ell}^{\romb})}{D}\normL{u^{\romb}-\mathcal{P}_{\ell}u^{\romb}}{D}.
\end{align*}
Consequently, we derive
\begin{align*}
\normL{\nabla(u-u_{{\rm ms},\ell})}{D}\leq \normL{\nabla(u^{\romb}-\mathcal{P}_{\ell}u^{\romb})}{D}+
\mathrm{Pe}\normL{u^{\romb}-\mathcal{P}_{\ell}u^{\romb}}{D}.
\end{align*}
Finally, an application of \eqref{eq:glo-l2} and \eqref{eq:glo-energy} yields the desired assertion.
\end{proof}
Compared with the standard finite element estimate \eqref{eq:q1element-error}, Theorem \ref{prop:wavelet-basedconv} implies convergence with respect to the level parameter $\ell$ instead of $H$. In the case that $\epsilon\ll 1$,  \eqref{eq:q1element-error} becomes much worse since $\|u\|_{H^2(D)}$ scales as $\epsilon^{-\alpha}$ for some positive constant $\alpha$. For $d=1$, it is proved that $\|u\|_{H^2(D)}\lesssim \epsilon^{-3/2}$ \cite[Remark 3.4]{le2017numerical}. Moreover, we compare in Section \ref{sec:numer} the performance of our proposed method with the standard finite element method \eqref{eq:q1element-error} which demonstrates consistently the former attains more accuracy.
\section{Numerical tests}\label{sec:numer}
In this section, several numerical experiments are presented to illustrate the computational performance of Algorithm \ref{algorithm:wavelet}.
In our experiments, we set the computational domain to be $D:=[0,1]^d$ for $d=2$ and $3$ and the constant force is employed, namely $f:=1$. Let $\mathcal{T}^{H}$ be a regular quasi-uniform rectangular mesh over $D$ with maximal mesh size $H$ and let $\mathcal{T}^{h}$ be a regular quasi-uniform rectangular mesh over each coarse element $T\in \mathcal{T}^{H}$ with maximal mesh size $h$.
To estimate the accuracy of the numerical solution $u_{\text{ms},\ell}$ \eqref{eq:u-ms}, we calculate the relative $L^2(D)$-error and the relative \noteLi{semi-}$H^1(D)$-error defined by
\begin{equation*}
	e_{L^2}:=\frac{\|u_{\text{ms},\ell}-u_{\text{ref}}\|_{L^2(D)}}{\| u_{\text{ref}}\|_{L^2(D)}},\quad
	e_{H^1}:=\workLi{\frac{\|\nabla (u_{\text{ms},\ell}- u_{\text{ref}})\|_{L^2(D)}}{\|\nabla u_{\text{ref}}\|_{L^2(D)}}},
\end{equation*}
where $u_{\text{ref}}$ is the reference solution computed using classical $\mathbb{Q}^1$ conforming Galerkin over $\mathcal{T}^{h}$ with $h:=\sqrt{2}/2^{10}$ for $d=2$ and $h:=\noteLi{\sqrt{3}}/2^6$ for $d=3$. \noteLi{Here, $h$ denotes the diameter of the largest element.}

We will compare our approach with SUPG for 2-d Examples 1-4. Let us briefly review
the SUPG model to \eqref{eq:original} \cite{FRANCA1992253}.
Let $(\bullet,\bullet)_T:=(\bullet,\bullet)_{L^2(T)}$ denote the 
$L^2$ scalar product over an element $T$ and $\|\bm{b}\|_{L^{\infty}(T)}$ being the essential supremum of $\bm{b}$ over $T$ for $T\in\mathcal{T}^H$. Then SUPG seeks $u_H\in V_H$ such that 
\begin{align}\label{eq:SUPG}
B_{\mathrm{SUPG}}(u_H^{\mathrm{SUPG}},v_H)=F_{\mathrm{SUPG}}(v_H) 
  \qquad \text{for all }v_H\in V_H
\end{align}
with 
\begin{align*}
B_{\mathrm{SUPG}}(u_H^{\mathrm{SUPG}},v_H)=a(u_H^{\mathrm{SUPG}},v_H)
    +\delta_\mathrm{SUPG}\sum\limits_{T\in \mathcal{T}^H} 
         (\bm{b}\cdot\nabla u_H^{\mathrm{SUPG}},\bm{b}\cdot\nabla v_H)_{T}
\end{align*}
and 
\begin{align*}
F_{\mathrm{SUPG}}(v_H)
   =( f,v_H)_{D}
 +\delta_\mathrm{SUPG}\sum\limits_{k\in \mathcal{T}^H} (f,\bm{b}\cdot\nabla v_H)_{T}.
\end{align*}
Here,  $\delta_\mathrm{SUPG}$ indicates the stability parameter, and we 
choose
 \[
\delta_\mathrm{SUPG} = \frac{H^2}{2\sqrt{2}\epsilon\max(12/\sqrt{2},H\|\bm{b}\|_{L^{\infty}(T)}/\epsilon)}
\]
in our numerical test. 

Note that the performance of SUPG method \eqref{eq:SUPG} is sensitive to the selection of the parameter $\delta_\mathrm{SUPG}$, and it can only improve the accuracy outside of the layers but not near the layers as observed in, e.g., \cite{le2017numerical}. 

\subsection*{Example 1}
We take the perturbation parameter $\epsilon:=10^{-2}$ and the velocity field is assumed to be
\begin{align*}
\bm{b}:=\textcolor{black}{\beta}\left[\sin(k\pi x) \cos(k\pi y),-\cos(k\pi x) \sin(k\pi y)\right]^T,
\end{align*}
which has a cellular structure with several eddies and separatrices. We consider two pairs of parameter,
\begin{align*}
(\textcolor{black}{\beta},k):=(2,24) \text{ and }(\textcolor{black}{\beta},k):=(8,48).
\end{align*}
The errors provided by \cref{algorithm:wavelet} with different coarse mesh size $H$ and different level parameter $\ell$ is displayed in \cref{ta:cv1_24b} and \cref{ta:cv1_48b}, respectively. \noteLi{Specifically}, we depict the reference solution and the multiscale solution $u_{\text{ms},\ell}$ from \cref{algorithm:wavelet} with $H=\sqrt{2}/16$ and $\ell=1$ in \cref{fig:cv2d_24_2} and \cref{fig:cv2d_48_8}.
\begin{table}[htbp]
	\centering
		\begin{adjustbox}{max width=\textwidth}
		\begin{tabular}{|c|c|c|c|c|c|c|c|c|c|c|c|}
			\hline
			\multirow{2}{*}{$H$}&\multirow{2}{*}{$\PeH$} & \multicolumn{2}{c|}{$\ell=0$} & \multicolumn{2}{c|}{$\ell=1$} & \multicolumn{2}{c|}{$\ell=2$}& \multicolumn{2}{c|}{FEM scheme \eqref{eqn:weakform_h}}& \multicolumn{2}{c|}{SUPG scheme \eqref{eq:SUPG}}\tabularnewline
			\cline{3-12}
			&& \specialcell{$e_{L^2}$} & \specialcell{$e_{H^1}$} & \specialcell{$e_{L^2}$} & \specialcell{$e_{H^1}$} & \specialcell{$e_{L^2}$} & \specialcell{$e_{H^1}$}&\specialcell{$e_{L^2}$} & \specialcell{$e_{H^1}$}&\specialcell{$e_{L^2}$} & \specialcell{$e_{H^1}$}\tabularnewline
			\hline
			$\sqrt{2}/8$ &50.0&0.83\% &4.44\% &0.32\%&2.03\% &0.19\%&1.06\% & 60.04\%& 79.39\%& 68.80\%& 82.27\% \tabularnewline
			\hline
			$\sqrt{2}/16$ &25.0&0.26\% &4.48\% &0.07\%&1.84\% &0.12\%&0.81\%  &61.55\%& 78.86\%& 47.08\%& 72.08\%  \tabularnewline
			\hline
			$\sqrt{2}/32$&12.5&0.28\% &7.19\% &0.03\%&2.09\% &0.0061\%&0.42\% & 10.32\%& 53.72\%& 21.09\%&53.00\% \tabularnewline
			\hline
			$\sqrt{2}/64$&6.25&0.11\% &6.28\% &0.01\%&1.08\% &0.0012\%&0.20\% & 0.88\%& 28.76\%& 4.26\%& 28.08\% \tabularnewline		
			\hline
		\end{tabular}
	\end{adjustbox}
	\caption{Errors provided by \cref{algorithm:wavelet},  \eqref{eqn:weakform_h} and \eqref{eq:SUPG} for Example 1 with $(\beta,k)=(2,24)$.}
	\label{ta:cv1_24b}
\end{table}
We observe that both relative errors $e_{L^2}$ and $e_{H^1}$ decrease as level parameter $\ell$ increases as expected. Errors corresponding to the case $(\beta,k)=(8,48)$ are generally larger than $(\beta,k)=(2,24)$ since the former case has a larger P\'{e}clet number and thus results in a more oscillating solution as exhibited in \cref{fig:cv2d_24_2} and \cref{fig:cv2d_48_8}. We also observe from \cref{fig:cv2d_24_2} and \cref{fig:cv2d_48_8} that the multiscale solution $u_{\text{ms},\ell}$ with $H=\sqrt{2}/16$ and $\ell=1$ approximates the reference solution well and recover multiscale details. On the contrary, both \eqref{eqn:weakform_h} and \eqref{eq:SUPG} fail to generate accurate solutions since polynomials can not capture oscillating information of the velocity fields especially if 
$H>\sqrt{2}/k$.
\begin{figure}[htbp]
	\centering
	\subfigure[$u_{\text{ref}}$]{
		\includegraphics[trim={0cm 0 0cm 0},clip,width=2.3in]{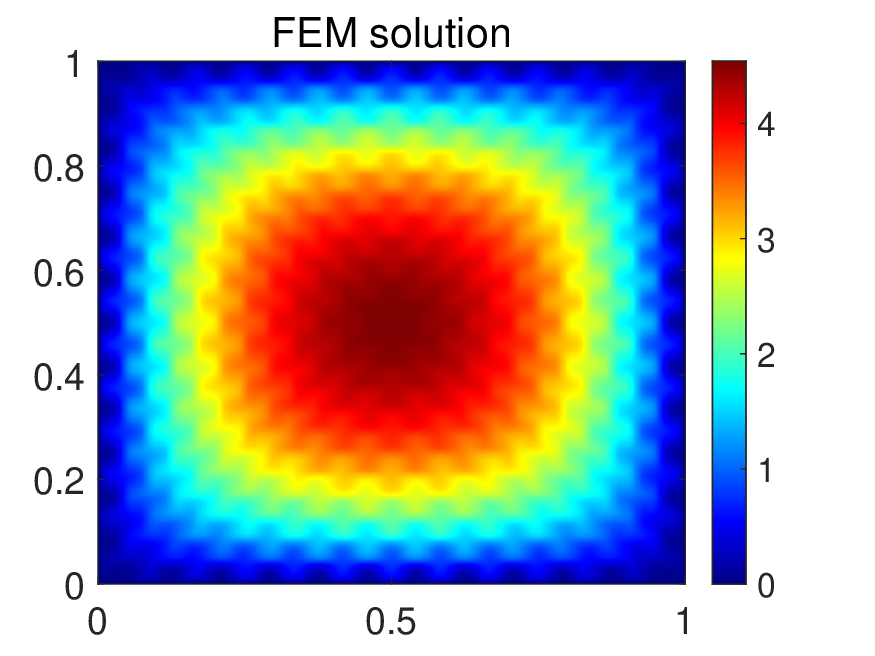}}
	\subfigure[$u_{\text{ms},\ell}$]{
		\includegraphics[trim={0cm 0 0cm 0},clip,width=2.3in]{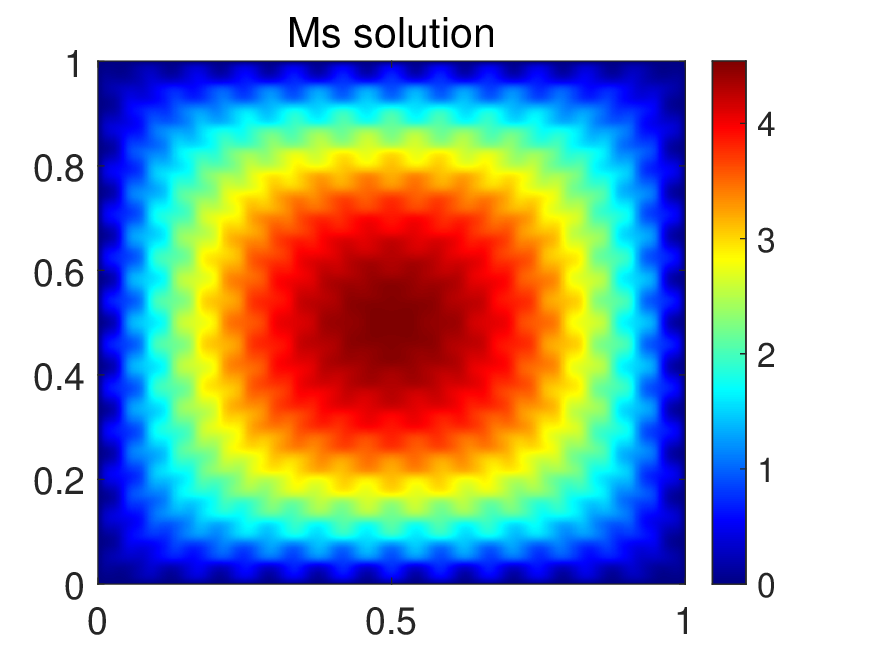}}
	\caption{Comparison of reference solution $u_{\text{ref}}$ with $u_{\text{ms},\ell}$ from \cref{algorithm:wavelet} with $H=\sqrt{2}/16$ and $\ell=1$ for Example 1 with $(\beta,k)=(2,24)$.}
	\label{fig:cv2d_24_2}
\end{figure}
\begin{table}[htbp]
	\centering
		\begin{adjustbox}{max width=\textwidth}
		\begin{tabular}{|c|c|c|c|c|c|c|c|c|c|c|c|}
			\hline
			\multirow{2}{*}{$H$}&\multirow{2}{*}{$\PeH$} & \multicolumn{2}{c|}{$\ell=0$} & \multicolumn{2}{c|}{$\ell=1$} & \multicolumn{2}{c|}{$\ell=2$}& \multicolumn{2}{c|}{FEM scheme \eqref{eqn:weakform_h}}& \multicolumn{2}{c|}{SUPG scheme \eqref{eq:SUPG}}\tabularnewline
			\cline{3-12}
			&& \specialcell{$e_{L^2}$} & \specialcell{$e_{H^1}$} & \specialcell{$e_{L^2}$} & \specialcell{$e_{H^1}$} & \specialcell{$e_{L^2}$} & \specialcell{$e_{H^1}$}&\specialcell{$e_{L^2}$} & \specialcell{$e_{H^1}$}&\specialcell{$e_{L^2}$} & \specialcell{$e_{H^1}$}\tabularnewline
			\hline
			$\sqrt{2}/8$&200&0.91\%& 3.79\%&0.32\%&1.09\%&0.30\%&0.91\% &141.0 \% & 119.7\%& 85.77\%& 94.02\%\tabularnewline
			\hline
			$\sqrt{2}/16$ &100&1.20\%&3.46\%&1.22\%&2.48\%&0.83\%&1.73\% & 143.3\% & 120.0\%& 73.40\%& 89.60\% \tabularnewline
			\hline
			$\sqrt{2}/32$&50&2.18\%& 5.47\%&0.82\%&2.58\%&0.37\%&1.43\% & 142.9\% & 119.8\%& 52.48\%& 83.35\%  \tabularnewline
			\hline
			$\sqrt{2}/64$&25&1.09\%& 8.62\%&0.05\%&3.57\%&0.013\%&0.74\%& 15.33\% & 77.39\%& 32.09\%& 71.46\%  \tabularnewline			
			\hline
		\end{tabular}
		\end{adjustbox}
\caption{Errors provided by \cref{algorithm:wavelet}, \eqref{eqn:weakform_h} and \eqref{eq:SUPG} for Example 1 with $(\textcolor{black}{\beta},k)=(8,48)$.}
	\label{ta:cv1_48b}
\end{table}
\subsection*{Example 2}
This example is adopted from \cite{calo2016multiscale}. We take $\epsilon=10^{-3}$ and define the velocity to be
\begin{align*}
\bm{b}:=\left[-\frac{\partial{g}}{\partial y}, \frac{\partial{g}}{\partial x}\right]^T \text{ with }
g(x,y):=\frac{1}{60\pi}\sin(5\pi x)\sin(6\pi y)+\frac{1}{\noteLi{200}}(x+y).
\end{align*}
\begin{figure}[htbp]
	\centering
	\subfigure[Reference FEM solution]{
		\includegraphics[trim={0cm 0 0cm 0},clip,width=2.3in]{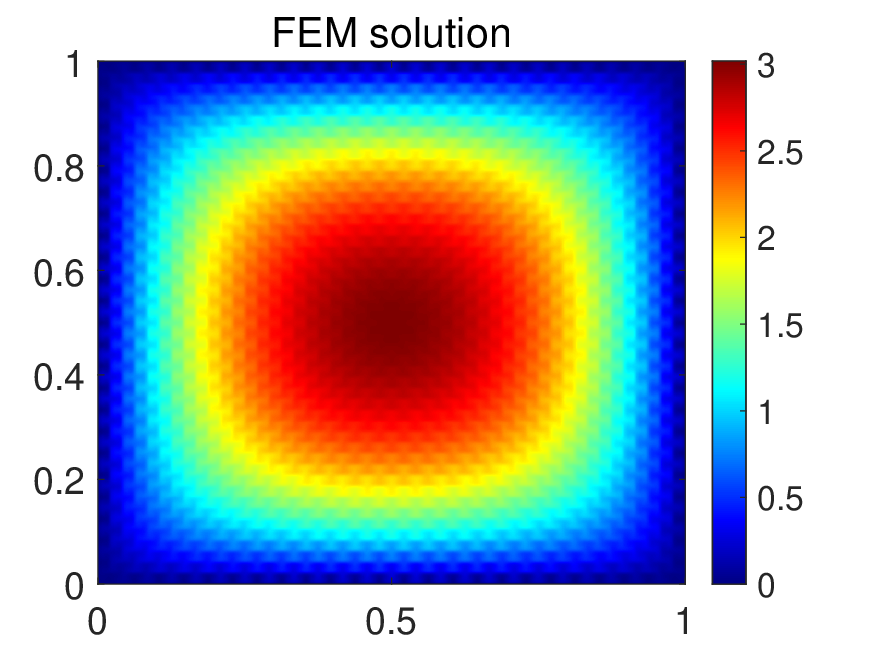}}
	\subfigure[WEMsFEM solution]{
		\includegraphics[trim={0cm 0 0cm 0},clip,width=2.3in]{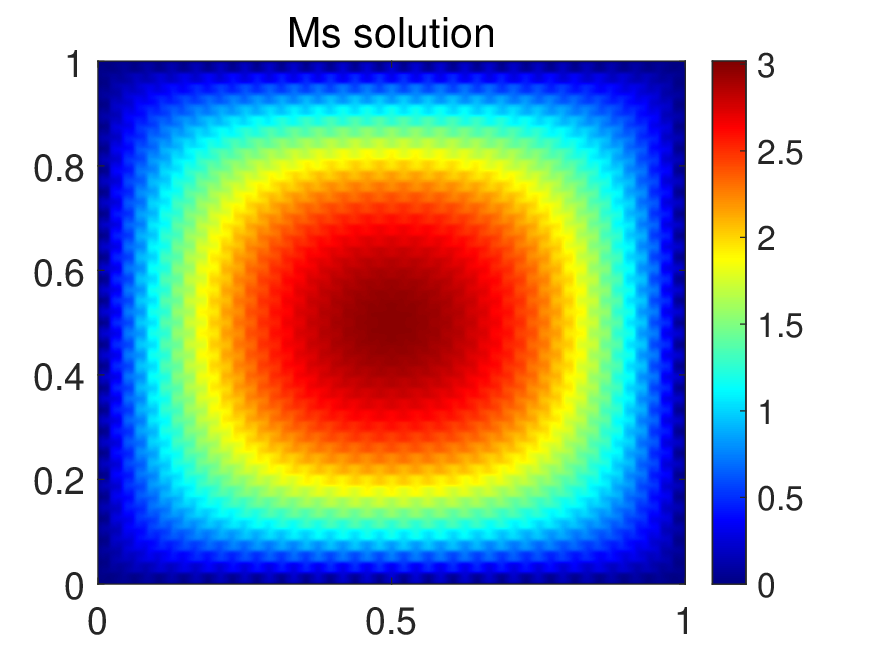}}
	\caption{Comparison of reference solution $u_{\text{ref}}$ with $u_{\text{ms},\ell}$ from \cref{algorithm:wavelet} with $H:=\sqrt{2}/16$ and $\ell:=1$ for Example 1 with $(\beta,k):=(8,48)$.}
	\label{fig:cv2d_48_8}
\end{figure}

\begin{table}[htbp]
	\centering
	\begin{adjustbox}{max width=\textwidth}
		\begin{tabular}{|c|c|c|c|c|c|c|c|c|c|c|c|}
			\hline
			\multirow{2}{*}{$H$}&\multirow{2}{*}{$\PeH$} & \multicolumn{2}{c|}{$\ell=0$} & \multicolumn{2}{c|}{$\ell=1$} & \multicolumn{2}{c|}{$\ell=2$}& \multicolumn{2}{c|}{FEM scheme \eqref{eqn:weakform_h}}& \multicolumn{2}{c|}{SUPG scheme \eqref{eq:SUPG}}\tabularnewline
			\cline{3-12}
			&& \specialcell{$e_{L^2}$} & \specialcell{$e_{H^1}$} & \specialcell{$e_{L^2}$} & \specialcell{$e_{H^1}$} & \specialcell{$e_{L^2}$} & \specialcell{$e_{H^1}$}&\specialcell{$e_{L^2}$} & \specialcell{$e_{H^1}$}&\specialcell{$e_{L^2}$} & \specialcell{$e_{H^1}$}\tabularnewline
			\hline
			$\sqrt{2}/8$ &26.25&3.20\% &7.93\% &0.19\%&2.07\% &0.03\%&0.42\%  & 12.18\%& 49.09\%& 16.23\%& 44.86\% \tabularnewline
			\hline
			$\sqrt{2}/16$ &13.12&0.61\% &6.01\% &0.05\%&1.37\% &0.0054\%&0.22\%   & 1.76\%& 23.06\%& 3.44\%& 22.88\%\tabularnewline
			\hline
			$\sqrt{2}/32$&6.56&0.10\% &2.91\% &0.0090\%&0.55\% &0.0011\%&0.10\%  & 0.42\%& 11.24\%& 0.78\%& 11.23\%\tabularnewline
			\hline
			$\sqrt{2}/64$&3.28&0.02\% &1.14\% &0.0017\%&0.20\% &0.0002\%&0.04\%  &  0.10\%& 5.62\%& 0.19\%& 5.62\% \tabularnewline		
			\hline
		\end{tabular}
\end{adjustbox}
\caption{Errors provided by \cref{algorithm:wavelet}, \eqref{eqn:weakform_h} and \eqref{eq:SUPG} for Example 2.}
	\label{ta:cv2}
\end{table}

\begin{figure}[htbp]
	\centering
	\subfigure[$u_{\text{ref}}$]{
		\includegraphics[trim={0cm 0 0cm 0},clip,width=2.3in]{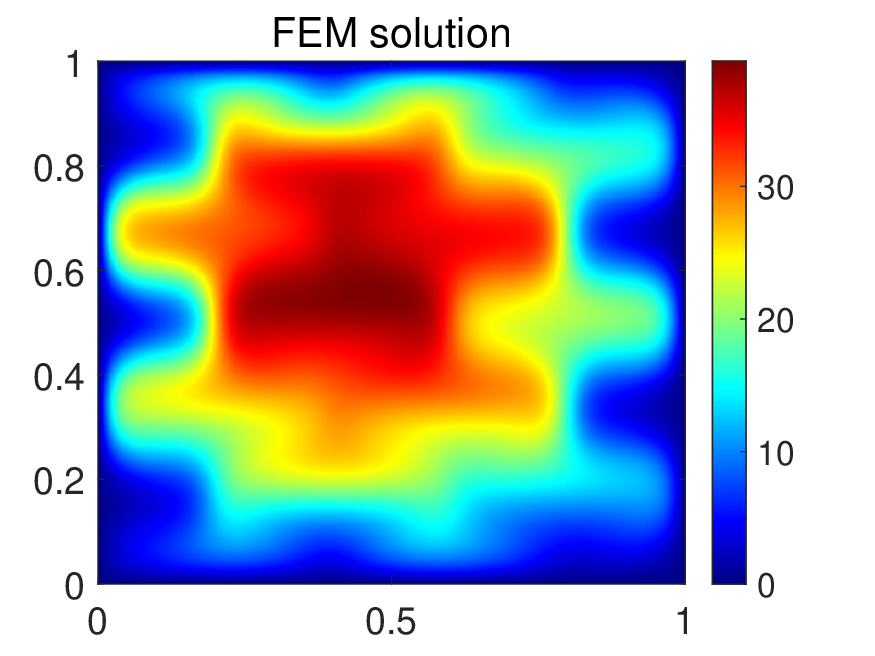}}
	\subfigure[$u_{\text{ms},\ell}$]{
		\includegraphics[trim={0cm 0 0cm 0},clip,width=2.3in]{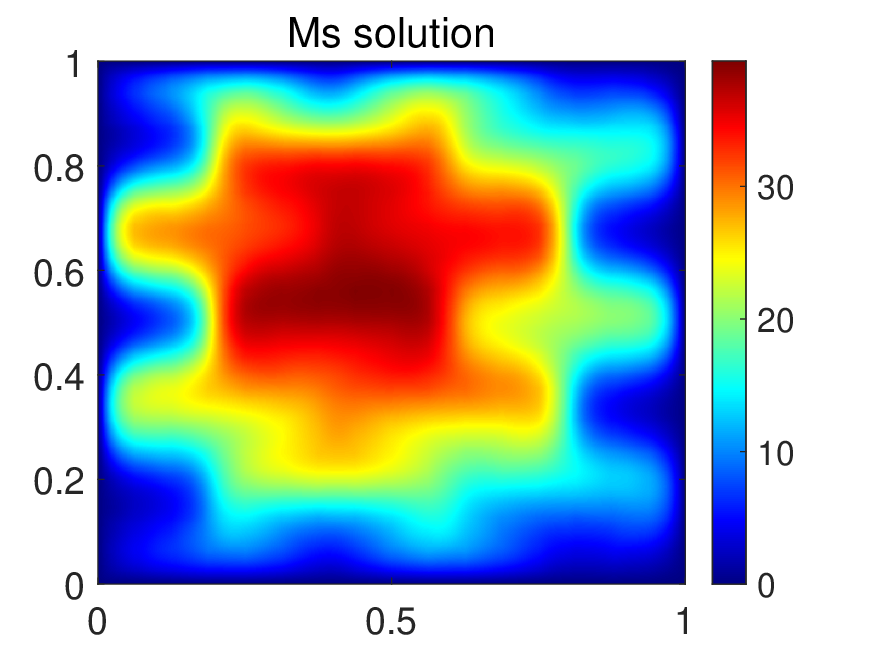}}
	\caption{Comparison of reference solution $u_{\text{ref}}$ with $u_{\text{ms},\ell}$ from Algorithm \ref{algorithm:wavelet} with $H:=\sqrt{2}/16$ and $\ell=1$ for Example 2.}
	\label{fig:ex2}
\end{figure}
This velocity field also has a cellular structure with eddies and channels as in  Example 1. The convergence of $u_{\text{ms},\ell}$ with different coarse-grid size $H$ and level parameter $\ell$ is reported in \cref{ta:cv2}. Convergence with respect to the level parameter $\ell$ can be observed in this example, we also observe convergence with respect to $H$. Surprisingly, \cref{algorithm:wavelet} with $\ell=0$ and $H=\sqrt{2}/8$ can deliver an accurate solution. The reference solution $u_{\text{ref}}$ and multiscale solution $u_{\text{ms},\ell}$ with $H:=\sqrt{2}/16$ and $\ell=1$ are displayed in \cref{fig:ex2}, which confirm the data in \cref{ta:cv2}. It can also be seen that $u_{\text{ms},\ell}$ with $H:=\sqrt{2}/16$ and $\ell=0$ can capture the microscale feature of the reference solution. For this example, the performance of \eqref{eqn:weakform_h} and \eqref{eq:SUPG} is not bad since the solution is relatively smooth.
\subsection*{Example 3}
\begin{table}[htbp]
	\centering
	\begin{adjustbox}{max width=\textwidth}
		\begin{tabular}{|c|c|c|c|c|c|c|c|c|c|c|c|}
			\hline
			\multirow{2}{*}{$H$}&\multirow{2}{*}{$\PeH$} & \multicolumn{2}{c|}{$\ell=0$} & \multicolumn{2}{c|}{$\ell=1$} & \multicolumn{2}{c|}{$\ell=2$}& \multicolumn{2}{c|}{FEM scheme \eqref{eqn:weakform_h}}& \multicolumn{2}{c|}{ SUPG scheme \eqref{eq:SUPG}}\tabularnewline
			\cline{3-12}
			&& \specialcell{$e_{L^2}$} & \specialcell{$e_{H^1}$} & \specialcell{$e_{L^2}$} & \specialcell{$e_{H^1}$} & \specialcell{$e_{L^2}$} & \specialcell{$e_{H^1}$}&\specialcell{$e_{L^2}$} & \specialcell{$e_{H^1}$}&\specialcell{$e_{L^2}$} & \specialcell{$e_{H^1}$}\tabularnewline
			\hline
			$\sqrt{2}/8$ &50.0&0.47\% &2.67\% &0.04\%&0.68\% &0.02\%&0.45\% & 39.83\%& 62.22\%& 71.57\%& 80.74\%  \tabularnewline
			\hline
			$\sqrt{2}/16$&25.0&0.16\% &1.95\% &0.03\%&1.31\% &0.02\%&1.10\%  & 41.20\%& 64.21\%& 51.37\%& 68.92\%  \tabularnewline
			\hline
			$\sqrt{2}/32$&12.50&0.25\% &4.48\% &0.08\%&2.49\% &0.01\%&0.54\% &36.83\%& 60.28\%& 23.09\%& 55.27\%\tabularnewline
			\hline
			$\sqrt{2}/64$&6.25&0.56\% &7.10\% &0.02\%&1.42\% &0.0014\%&0.27\% &12.55\%& 35.37\%& 3.15\%& 33.60\%  \tabularnewline		
			\hline
		\end{tabular}
	\end{adjustbox}
\caption{Errors provided by \cref{algorithm:wavelet}, \eqref{eqn:weakform_h} and \eqref{eq:SUPG} for Example 3.}
	\label{ta:cv3}
\end{table}

A channelized flow field without eddies is considered with $\epsilon:=1$ and 
$
\bm{b}:=200[\sin(48\pi y),0]^T.
$
The numerical results are presented in \cref{ta:cv3}, and similar convergence behavior as Example 1 is observed. We depict the reference solution $u_{\text{ref}}$ and $u_{\text{ms},\ell}$ from \cref{algorithm:wavelet} with $H:=\sqrt{2}/16$, $\ell=0$ in \cref{fig:ex3}. The errors displayed in \cref{ta:cv3} verify that  \cref{algorithm:wavelet} can provide an accurate solution.
In particular, even if the level parameter $\ell=0$, relative $L^2(D)$-errors and semi-$H^1(D)$-errors are all below $0.56\%$ and $7.10\%$, respectively.
Moreover, we observe that the accuracy of \eqref{eqn:weakform_h} and \eqref{eq:SUPG} for this case are not satisfactory, with relative semi-$H^1(D)$-errors above $30\%$ for both methods.

\begin{figure}[htbp]
	\centering
	\subfigure[$u_{\text{ref}}$]{
		\includegraphics[trim={0cm 0 0cm 0},clip,width=2.3in]{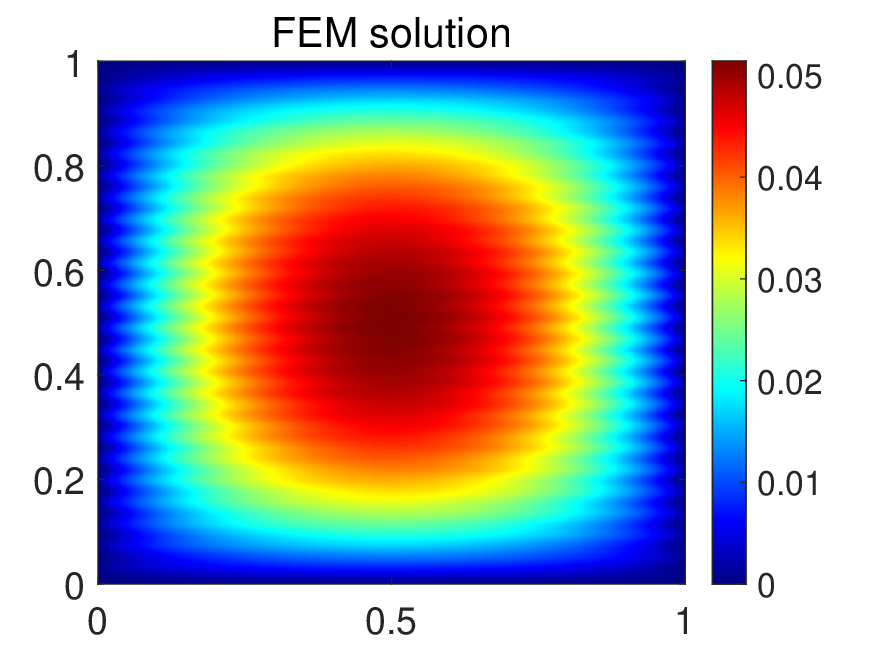}}
	\subfigure[$u_{\text{ms},\ell}$]{
		\includegraphics[trim={0cm 0 0cm 0},clip,width=2.3in]{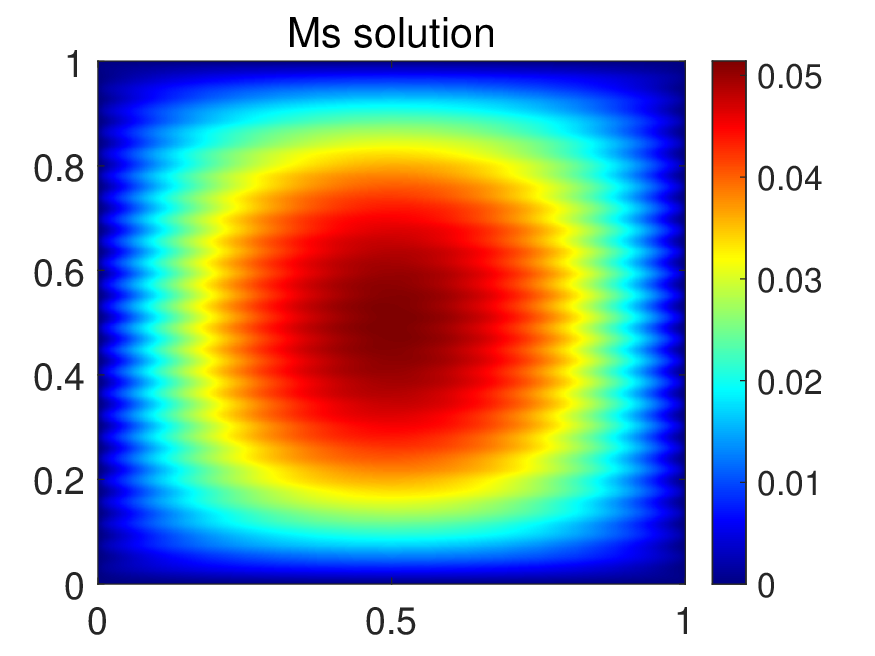}}
	\caption{Comparison of the reference solution $u_{\text{ref}}$ with $u_{\text{ms},\ell}$ from \cref{algorithm:wavelet} with $H:=\sqrt{2}/16$ and $\ell=0$ for Example 3.}
	\label{fig:ex3}
\end{figure}
\begin{table}[htbp]
	\centering
		\begin{adjustbox}{max width=\textwidth}
		\begin{tabular}{|c|c|c|c|c|c|c|c|c|c|c|c|}
			\hline
			\multirow{2}{*}{$H$}&\multirow{2}{*}{$\PeH$} & \multicolumn{2}{c|}{$\ell=0$} & \multicolumn{2}{c|}{$\ell=1$} & \multicolumn{2}{c|}{$\ell=2$}& \multicolumn{2}{c|}{FEM scheme \eqref{eqn:weakform_h}}& \multicolumn{2}{c|}{SUPG scheme \eqref{eq:SUPG}}\tabularnewline
			\cline{3-12}
			&& \specialcell{$e_{L^2}$} & \specialcell{$e_{H^1}$} & \specialcell{$e_{L^2}$} & \specialcell{$e_{H^1}$} & \specialcell{$e_{L^2}$} & \specialcell{$e_{H^1}$}&\specialcell{$e_{L^2}$} & \specialcell{$e_{H^1}$}&\specialcell{$e_{L^2}$} & \specialcell{$e_{H^1}$}\tabularnewline
			\hline
			$\sqrt{2}/8$ &32&0.74\% & 4.66\% & 0.14\% & 1.43\% & 0.04\% & 0.52\% &53.57\%&119.4\%& 32.88\%&91.48\% \tabularnewline
			\hline
			$\sqrt{2}/16$ &16&0.41\% & 4.73\% & 0.11\% & 1.55\% & 0.02\% & 0.38\% &22.5\%&99.44\%& 20.00\%& 82.37\%\tabularnewline
			\hline
			$\sqrt{2}/32$&8&0.18\% & 3.49\% & 0.04\% & 0.97\% & 0.0047\% & 0.20\% &8.88\%&72.93\%&10.75\%& 65.74\%\tabularnewline
			\hline
			$\sqrt{2}/64$&4& 0.06\% & 2.06\% & 0.0085\% & 0.41\% & 0.0010\% & 0.08\% &2.93\%&45.38\%& 3.66\%& 42.98\%\tabularnewline		
			\hline
		\end{tabular}
		\end{adjustbox}
\caption{Errors provided by \cref{algorithm:wavelet}, \eqref{eqn:weakform_h} and \eqref{eq:SUPG} for Example 4 with $\epsilon_1=1$.}
	\label{ta:cv4a}
\end{table}	

\subsection*{Example 4} 	
\begin{figure}[htbp]
	\centering
	\subfigure[$u_{\text{ref}}$]{
		\includegraphics[trim={0cm 0 0cm 0},clip,width=2.3in]{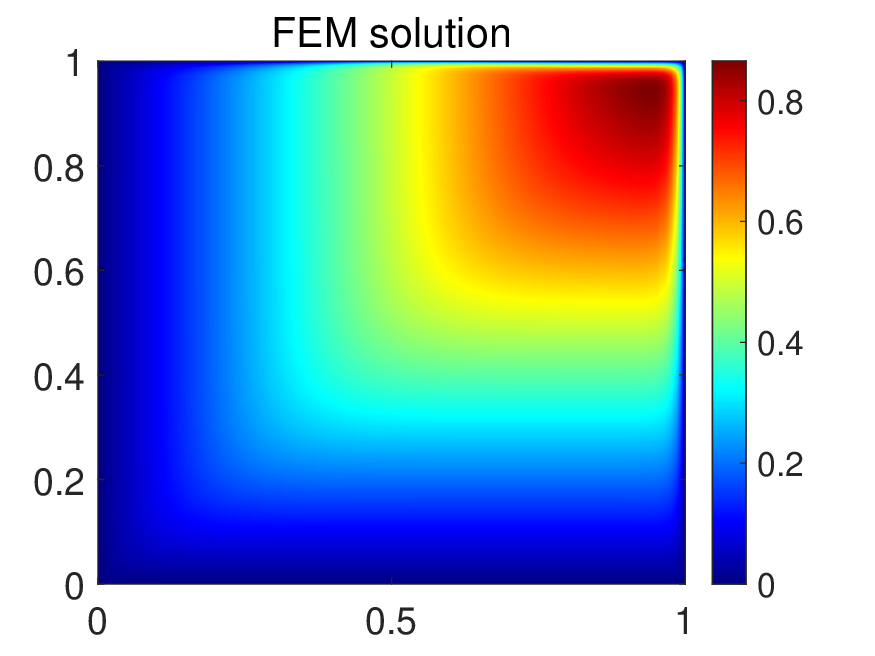}}
	\subfigure[$u_{\text{ms},\ell}$]{
		\includegraphics[trim={0cm 0 0cm 0},clip,width=2.3in]{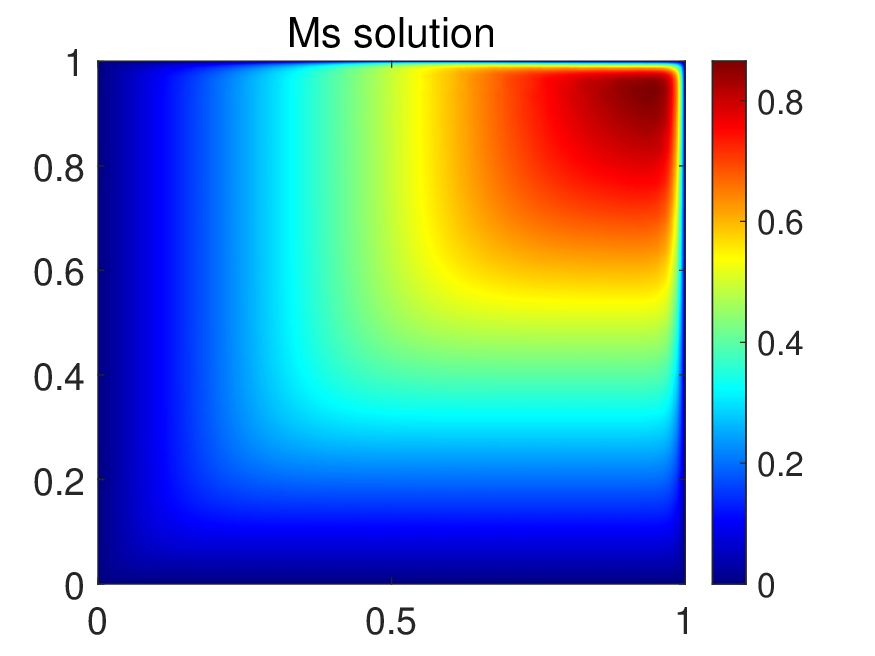}}
	\caption{Comparison of reference solution $u_{\text{ref}}$ with $u_{\text{ms},\ell}$ from  \cref{algorithm:wavelet} with $H:=\sqrt{2}/16$ and $\ell=1$ for Example 4 with $\epsilon_1=1$.}
	\label{fig:ex4a}
\end{figure}

This example is adopted from \cite{le2017numerical}. In particular, we consider a diffusion coefficient with multiple scales, such that 
we aim to solve 
\begin{align*}
-\nabla\cdot(A(x,y)\nabla u)+\bm{b}\cdot\nabla u=f
\end{align*} 
with the diffusion coefficient $A(x,y)$ being
\begin{eqnarray*}
	A(x,y)=\epsilon\bigg(1+0.5\cos\bigg(\frac{2\pi}{\epsilon_1}x\bigg)\bigg).
\end{eqnarray*}
Here, the oscillation parameter $\epsilon:=\frac{1}{128}$ and we take a constant velocity $\bm{b}:=(1,1)^T$. To guarantee the accuracy of reference solution $u_{\text{ref}}$, the fine scale mesh size $h$ is set to be $\sqrt{2}/2^{12}$.  

\begin{table}[htbp]
	\centering
	\begin{adjustbox}{max width=\textwidth}
		\begin{tabular}{|c|c|c|c|c|c|c|c|c|c|c|c|}
			\hline
			\multirow{2}{*}{$H$}&\multirow{2}{*}{$\PeH$} & \multicolumn{2}{c|}{$\ell=0$} & \multicolumn{2}{c|}{$\ell=1$} & \multicolumn{2}{c|}{$\ell=2$}& \multicolumn{2}{c|}{FEM scheme \eqref{eqn:weakform_h}} & \multicolumn{2}{c|}{SUPG scheme \eqref{eq:SUPG}}\tabularnewline
			\cline{3-12}
			&& \specialcell{$e_{L^2}$} & \specialcell{$e_{H^1}$} & \specialcell{$e_{L^2}$} & \specialcell{$e_{H^1}$} & \specialcell{$e_{L^2}$} & \specialcell{$e_{H^1}$} & \specialcell{$e_{L^2}$} & \specialcell{$e_{H^1}$} & \specialcell{$e_{L^2}$} & \specialcell{$e_{H^1}$}\tabularnewline
			\hline
			$\sqrt{2}/8$ &32&0.85\% & 5.33\% & 0.18\% & 1.82\% &  0.05\% & 0.77\%  & 54.84\%&  \color{black}{123.85}\%& 34.15\%& 93.66\%\tabularnewline
			\hline
			$\sqrt{2}/16$ &16&0.54\% & 6.83\% & 0.13\% & 2.08\% &  0.02\% & 0.52\%  &25.04\%&107.68\%&21.46\%&86.96\%   \tabularnewline
			\hline
			$\sqrt{2}/32$&8&0.40\% & 8.27\% & 0.08\% & 1.97\%  & 0.0075\% & 0.36\%  &10.05\%&83.79\%&12.00\%&73.47\% \tabularnewline
			\hline
			$\sqrt{2}/64$&4&0.21\% & 6.51\% & 0.02\% & 1.37\%  & 0.0021\% & 0.16\%  &3.11\%&53.82\%&4.30\%&51.16\%  \tabularnewline		
			\hline
		\end{tabular}
		\end{adjustbox}
\caption{Errors provided by \cref{algorithm:wavelet}, \eqref{eqn:weakform_h} and \eqref{eq:SUPG} for Example 4 with $\epsilon_1=1/64$.}
	\label{ta:cv4b}
\end{table}		
If $\epsilon_1=1/64$ and $H=\noteLi{\sqrt{2}}/16$, \noteLi{then the relative $L^2(D)$-error and the relative $H^1(D)$-error for Stab-MsFEM proposed in \cite{le2017numerical} are 23\% and 87\%, which are not reliable, especially near the boundary layer.} {{To make our comparison relevant, we implement Stab-MsFEM and obtain that the relative $L^2(D)$-error and the relative $H^1(D)$-error are \noteLg{\textcolor{black}{22}\% and 87\%}, which agree with the data reported in \cite{le2017numerical}. Remarkably, Tables \ref{ta:cv4a} and \ref{ta:cv4b} clearly demonstrate that Algorithm \ref{algorithm:wavelet} yields more accuracy even with the level parameter $\ell=0$, both near boundary layer and away from the boundary layer. \noteLi{In particular, if $\epsilon_1=1/64$, $H=\sqrt{2}/16$ and $\ell=0$, the relative $L^2(D)$-error and the relative semi-$H^1(D)$- error are $0.54\%$ and $6.83\%$, respectively}. 
\begin{figure}[htbp]
	\centering
	\subfigure[$u_{\text{ref}}$]{
		\includegraphics[trim={0cm 0 0cm 0},clip,width=2.3in]{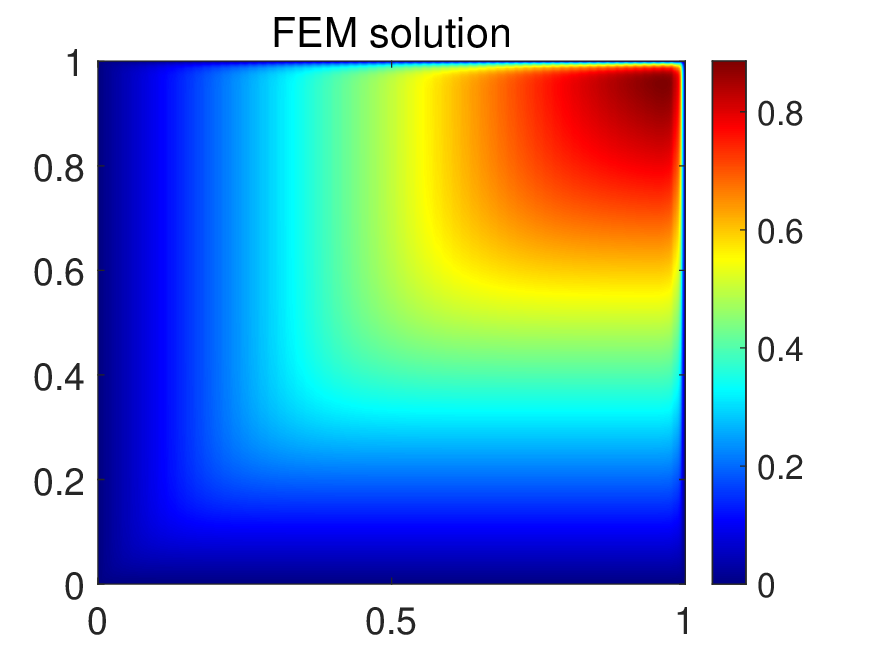}}
	\subfigure[$u_{\text{ms},\ell}$]{
		\includegraphics[trim={0cm 0 0cm 0},clip,width=2.3in]{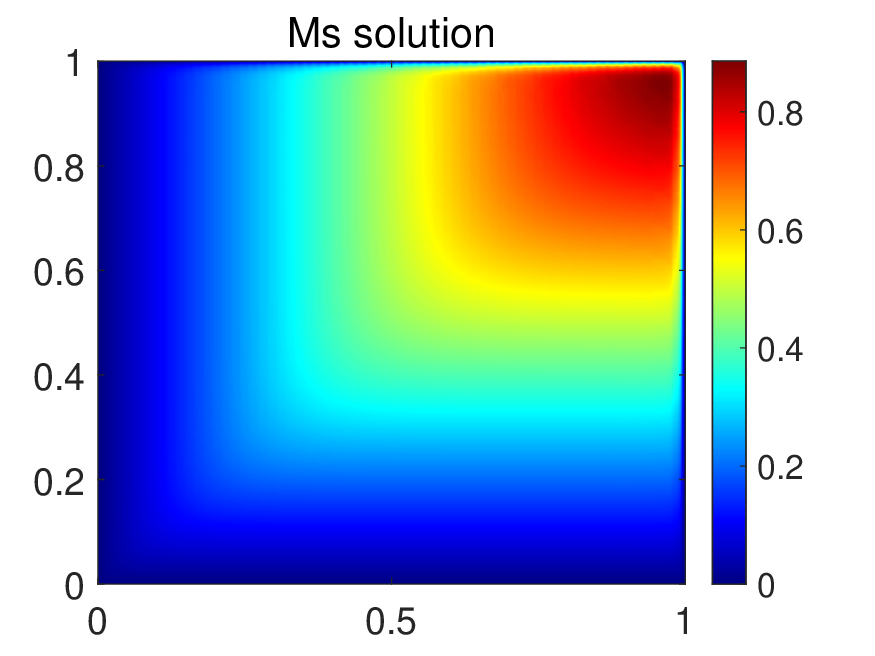}}
	\caption{Comparison of reference solution $u_{\text{ref}}$ with $u_{\text{ms},\ell}$ from  \cref{algorithm:wavelet} with $H:=\sqrt{2}/16$ and $\ell=1$ for Example 4 with $\epsilon_1=1/64$.}
	\label{fig:ex4b}
\end{figure}
Aligning with \cite{le2017numerical}, we further split the relative semi-$H^1(D)$-error into its layer contribution and out of the layer contribution, by defining the region for the layers as
$D_{\text{layer}}:=\left((0,1)\times(1-\delta_{\text{layer}},1)\right)\cup \left((1-\delta_{\text{layer}},1)\times (0,1)\right)
$. Here, we take $\delta_{\text{layer}}={2}/\text{Pe} \log(\text{Pe}/2)$. Then
the relative semi-$H^1_{\text{in}}$-error and relative semi-$H^1_{\text{out}}$-error are 
\begin{equation*}
e_{H^1_{\text{in}}}:=\workLi{\frac{\|\nabla (u_{\text{ms},\ell}- u_{\text{ref}})\|_{L^2(D_{\text{layer}})}}{\|\nabla u_{\text{ref}}\|_{L^2(D)}}}\quad\text{and}\quad
	e_{H^1_{\text{out}}}:=\workLi{\frac{\|\nabla (u_{\text{ms},\ell}- u_{\text{ref}})\|_{L^2(D\backslash D_{\text{layer}})}}{\|\nabla u_{\text{ref}}\|_{L^2(D)}}}.
\end{equation*}
This definition indicates the identity $e_{H^1}^2=e_{H^1_{\text{in}}}^2 +e_{H^1_{\text{out}}}^2$, i.e., the relative semi-$H^1(D)$-error is an upper bound for the relative semi-$H^1_{\text{in}}$-error and relative semi-$H^1_{\text{out}}$-error.

If $\epsilon_1=1/64$, $H=\sqrt{2}/16$ and $\ell=0$, then $e_{H^1_{\text{in}}}$ and $e_{H^1_{\text{out}}}$ are $6.5\%$ and $1.7\%$, respectively. \noteLg{In contrast, they are 87\% and 4\% for Stab-MsFEM proposed in \cite{le2017numerical} under the relative $H^1(D)$-error instead of relative semi-$H^1(D)$-error, which has a very slight difference in our test.} This clearly demonstrates that our proposed algorithm maintains high approximation property both in the layers and out of the layers.

Furthermore, we depict in \cref{fig:ex4a} and \cref{fig:ex4b} the reference solution $u_{\text{ref}}$ and multiscale solution $u_{\text{ms},\ell}$ with $H:=\sqrt{2}/16$ and \workLi{$\ell=1$}. We observe that the multiscale solution $u_{\text{ms},\ell}$ with $H:=\sqrt{2}/16$ and $\ell=1$ can capture the boundary layer accurately.
\subsection*{Example 5}
Example 5 is a 3-d example with $\epsilon:=1/16$, and
the velocity field being 
\[
\bm{b}:=[g_2-g_3,g_3-g_1,g_1-g_2]^T \text{ with }
\]
\begin{align*}
	g_1:=&\beta k \cos(kx)\sin(ky)\sin(kz), \\
	g_2:=&\textcolor{black}{\beta} k \sin(kx)\cos(ky)\sin(kz), \\
	g_3:=&\textcolor{black}{\beta} k \sin(kx)\sin(ky)\cos(kz).
\end{align*}
Different combinations of $(\beta,k)$ are tested and simulation results are reported in \cref{ta:cv5}. We only test  \cref{algorithm:wavelet} with $\ell=0$ due to the huge computational complexity involved for $d=3$. To adapt \cref{algorithm:wavelet} for $d=3$ and $\ell=0$, we need to define the coarse neighborhood and hierarchical bases up to $\ell=0$. 
The coarse neighborhood $\omega_i$ for each coarse node $i$ is a cube with this coarse node as its center and the hierarchical bases corresponds to level $\ell=0$ are eight nodal basis functions for $\omega_i$.
\begin{table}[htbp]
	\centering
	\begin{adjustbox}{max width=\textwidth}
		\begin{tabular}{|c|c|c|c|c|c|c|c|c|c|c|c|}\hline
			\multirow{2}{*}{	$(\textcolor{black}{\beta},k)$} & \multicolumn{4}{c|}{$H=\noteLi{\sqrt{3}}/{8}$} & \multicolumn{4}{c|}{$H=\noteLi{\sqrt{3}}/{16}$} \tabularnewline
			\cline{2-9}
			&$\PeH$& \specialcell{$e_{L^2}$} & \specialcell{$e_{H^1}$}  & $T_{\text{solve}}$ &$\PeH$& \specialcell{$e_{L^2}$} & \specialcell{$e_{H^1}$}  & $T_{\text{solve}}$  \tabularnewline\hline
			(1,4$\pi$)&61.49& 0.92\%  &7.2\% &0.8&30.75&0.29\%& 3.45\%&21.5  \tabularnewline\hline
			(1,8$\pi$)&122.98&1.38\%  &6.05\% &0.8&61.49&0.92\%& 11.06\%&22.3  \tabularnewline\hline			
			(4,4$\pi$)&245.95 & 3.72\% &21.36\% &0.8&122.98&1.55\%& 13.69\%&22.1 \tabularnewline\hline
			(4,8$\pi$)&491.9& 2.11\% &7.33\% &  0.8& 245.95 &3.71\%& 31.66\%&22.5  \tabularnewline\hline
		\end{tabular}
	\end{adjustbox}
	\caption{Errors provided by \cref{algorithm:wavelet} and CPU time in seconds for solving linear system in Example 5.}
	\label{ta:cv5}
\end{table}
\begin{figure}[htbp]
	\centering
	\includegraphics[trim={2cm 1.7cm 1.6cm 1.8cm},clip,width=3.3in]{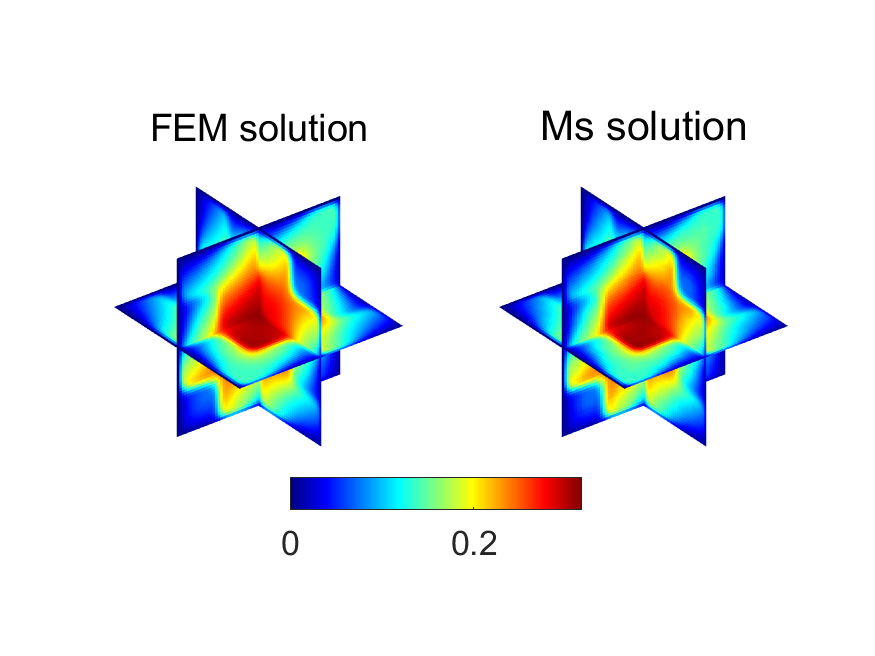}			
	\caption{Comparison of reference solution $u_{\text{ref}}$ with $u_{\text{ms},\ell}$ from \cref{algorithm:wavelet} with $\ell=0$ for Example 5 with $(\beta,k)=(1,4\pi)$.}
	\label{fig:cv3d_41}
\end{figure}

\begin{figure}[htbp]
	\centering
	\includegraphics[trim={2cm 1.7cm 1.6cm 1.8cm},clip,width=3.3in]{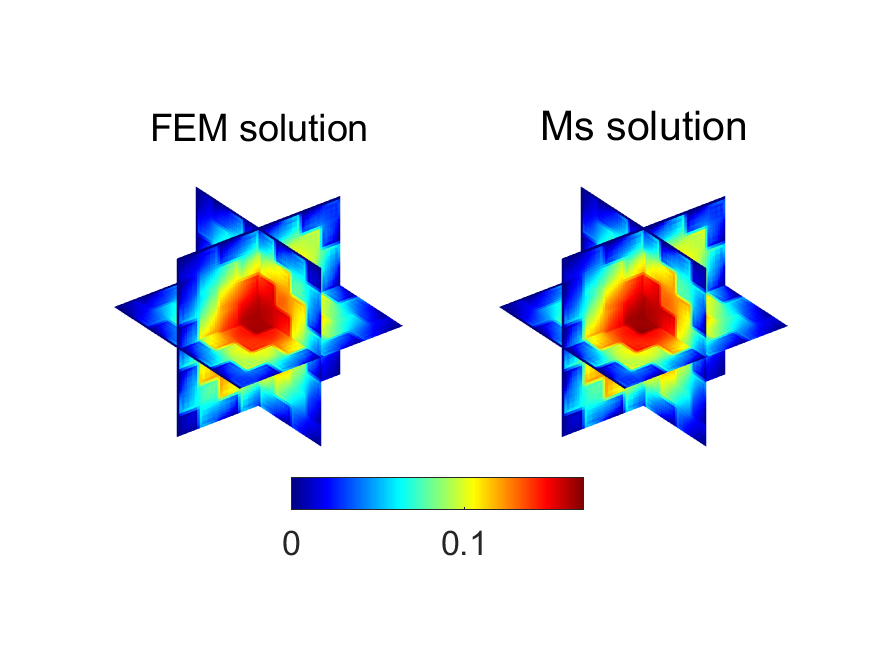}			
	\caption{Comparison of reference solution $u_{\text{ref}}$ with $u_{\text{ms},\ell}$ from  \cref{algorithm:wavelet} with $\ell=0$ for Example 5 with $(\beta,k)=(4,8\pi)$.}
	\label{fig:cv3d_84}
\end{figure}
We can see the performance of \cref{algorithm:wavelet} can provide an accurate solution in most cases. For example, the $L^2(D)$-relative error is merely $3.17\%$ when $H=\sqrt{3}/{8}$ and $(\beta,k)=(4,8\pi)$. 
Moreover, \cref{fig:cv3d_41} and \cref{fig:cv3d_84} demonstrate that \cref{algorithm:wavelet} with $\ell=0$ can capture the microscale feature of the reference solution. We note that the CPU time for applying a
direct solver for solving the linear system resulting from FEM is about 147 seconds regardless of $\alpha$ and $k$. In comparison, the CPU time for \cref{algorithm:wavelet} are about 0.8 seconds and 22 seconds for $H=\sqrt{3}/{8}$ and $H=\sqrt{3}/{16}$ respectively, and hence huge computational cost is saved.

\subsection{Sensitivity with respect to the mesh P\'eclet number}

Next, we set the level parameter $\ell:=0$ and coarse mesh size $H:=\sqrt{2}/16$ in \cref{algorithm:wavelet} and test it with different values of oscillation parameter $\epsilon$ for Examples 1-4 to study its
robustness with respect to the mesh P\'eclet number $\PeH$. Experiments results are provided in \cref{fig:peh}. It is observed relative errors generally increase as mesh P\'eclet number $\PeH$ becomes larger.
However, for examples 2-4, relative errors are quite stable for $\PeH=\mathcal{O}(10)$. For example 1, although
accuracy of \cref{algorithm:wavelet} have a relative stronger dependence on the mesh P\'eclet number $\PeH$, the relative errors are still very small even if $\PeH$ is as large as 250 despite $\ell=0$.

		\begin{figure}[htbp]
		\centering
		\includegraphics[trim={17cm 4cm 18cm 2cm},clip,width=.75\textwidth]{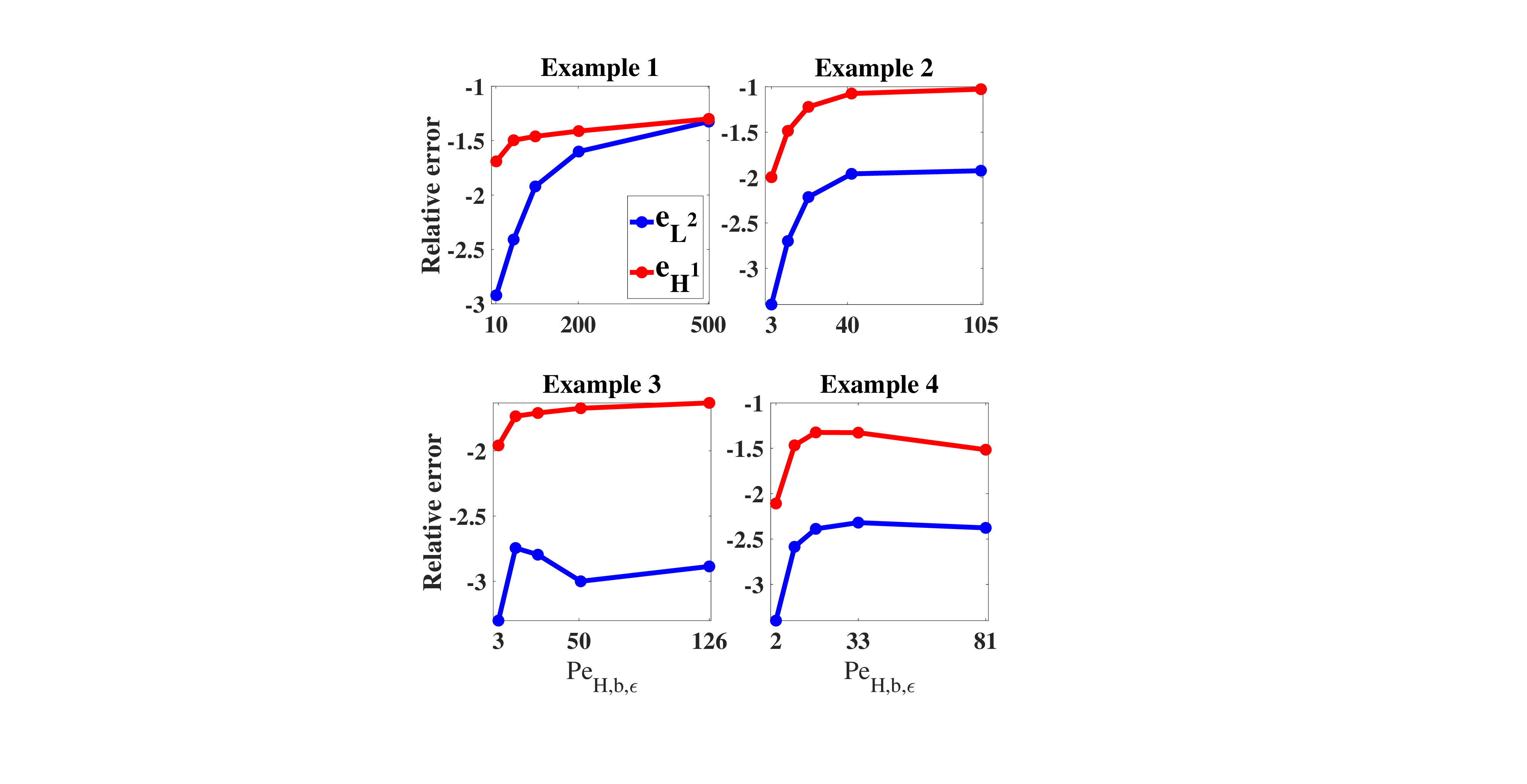}
		\caption{{Performance of \cref{algorithm:wavelet} with $\ell=0$ and $H=\sqrt{2}/16$ against the mesh P\'eclet number $\PeH$. Here, $(\textcolor{black}{\beta},k)=(8,48)$ in Example 1 and $\epsilon_1:=1/64$ in Example 4.}}
		\label{fig:peh}
	\end{figure}
\subsection{Accuracy with respect to the wavelet level $\ell$}
Finally, we investigate the performance of \cref{algorithm:wavelet} with respect to the level parameter $\ell$ by Examples 1-4. The results are depicted in \cref{fig:l}, and exponential decay rate is observed.
		\begin{figure}[htbp]
	\centering
	\includegraphics[trim={16cm 2cm 16cm 2cm},clip,width=0.75\textwidth]{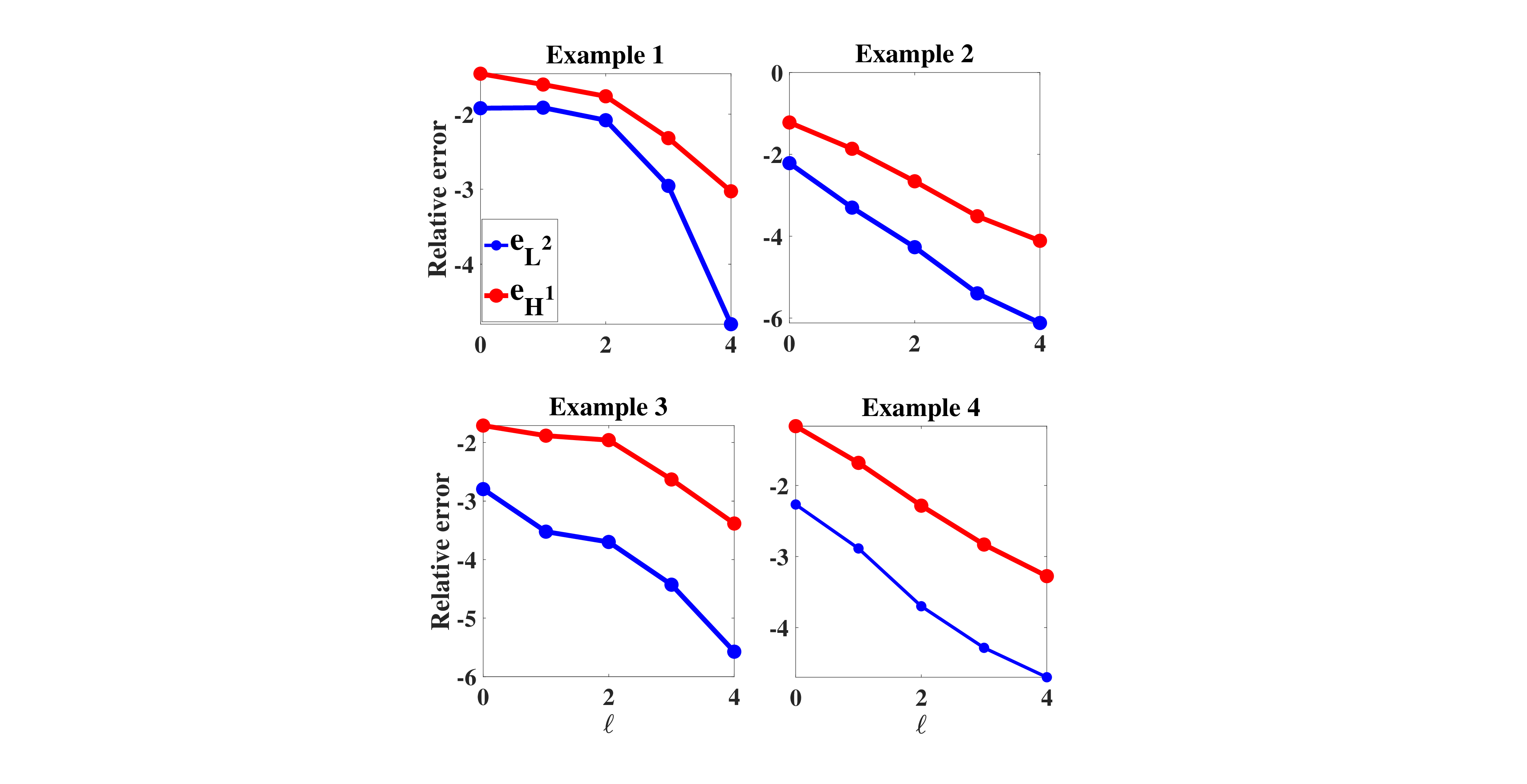}
	\caption{\textcolor{black}{{Performance of \cref{algorithm:wavelet} with $H=\sqrt{2}/16$ against the wavelet level $\ell$. Here, $(\textcolor{black}{\beta},k)=(8,48)$ in Example 1 and $\epsilon_1:=1/64$ in Example 4. The errors are displayed in log-10 scale.}}}
	\label{fig:l}
\end{figure}

\section{Conclusions}\label{sec:conclusion}
We develop a novel Wavelet-based Edge Multiscale Finite Element Method (WEMsFEM) to solve singularly perturbed convection-diffusion problems. The key feature of this approach is to introduce an edge multiscale ansatz space which has guaranteed approximation properties both inside and outside of the layers. The construction of this edge multiscale ansatz space is motivated by the local and global splitting of the solution, and their basis functions are solved locally using hierarchical bases as the Dirichlet data. The convergence of WEMsFEM with respect to the level parameter in the hierarchical bases is rigorously justified, which is confirmed numerically for a wide range of 2-d or 3-d convection-dominated diffusion problems with large P\'{e}clet number, including high oscillation and boundary layers.

\textcolor{black}{Note that the construction of the edge multiscale ansatz space involves solving local problems, and standard stabilized methods such as SUPG on regular quasi-uniform mesh are utilized for the sake of simplicity. Consequently, our proposed WEMsFEM can not handle the case with extremely large P\'{e}clet number if no additional mesh or multiscale ansatz space adaptivivity is developed. We will investigate in the future on the coarse grid adaptivity and local multiscale space adaptivity to handle such challenging issues.}

Moreover, we only consider an incompressible velocity in this paper. If the velocity ${\bm b}\in L^{\infty}(D;\mathbb{R}^d)$ is compressible, then the bilinear form $a(\cdot,\cdot): V\times V\to \mathbb{R}$ is possibly no longer coercive in the case of large P\'{e}clet number we are interested in. Under such scenario, we have to replace the Galerkin orthogonality with Schatz argument to derive the quasi-optimality established in Theorem 5.2. We expect our proposed method should also work in such case, and rigorous derivation will be carried out in the future. 
\appendix
\section{Very-weak solutions to singularly perturbed convection-diffusion problems}
\label{sec:appendix}
In this appendix, we derive an $L^2(\omega_i)$-estimate for boundary valued singularly perturbed convection-diffusion problem \noteLg{with $\omega_i\subset\mathbb{R}^{d}$ being a bounded open domain with a Lipschitz boundary for $d=2,3$}, which plays a crucial role in the error analysis.
Let $\omega_i$ be a coarse neighborhood for any
$i=1,\cdots,N$. For any $g\in L^2(\partial\omega_i\backslash\partial D)$, we define the following convection diffusion problem
\begin{equation}\label{eq:pde-very}
\left\{\begin{aligned}
\mathcal{L}_i v:=-\epsilon\Delta v+\bm{b}\cdot\nabla v&=0 && \text{ in } \omega_i,\\
v&=g &&\text{ on }\partial \omega_i\backslash\partial D\\
v&=0 &&\text{ on }\partial \omega_i\cap\partial D.
\end{aligned}\right.
\end{equation}
Our goal is to derive the $L^2(\omega_i)$-estimate of the solution $v$. To this end, we employ a nonstandard variational form in the spirit
of the transposition method \cite{MR0350177}. Denote the dual operator of $\mathcal{L}_i$ as $\mathcal{L}_i^*$. Since $\nabla\cdot \bm{b}=0$, we obtain $\mathcal{L}_i^*=-\epsilon\Delta-\bm{b}\cdot\nabla$. Let the test space $X(\omega_i)\subset H^1_{0}(\omega_i)$ be defined by
\begin{align}\label{eq:test-space}
X(\omega_i):=\{z:\mathcal{L}_i^*z\in L^2(\omega_i)\text{ and } z\in H^1_{0}(\omega_i)\}.
\end{align}
This test space $X(\omega_i)$ is endowed with the norm $\|\cdot\|_{X(\omega_i)}$:
\[
\forall z\in X(\omega_i):\|z\|_{X(\omega_i)}^2=\int_{\omega_i}|\nabla z|^2\dx+\int_{\omega_i}|\mathcal{L}_i^*z|^2\dx.
\]
Then we propose the following weak formulation corresponding to Problem \eqref{eq:pde-very}:
seeking $v\in L^2(\omega_i)$ such that
\begin{align}\label{eq:nonstd-variational}
\int_{\omega_i}v\mathcal{L}_i^{*}( z)\dx=-\epsilon\int_{\partial \omega_i\backslash\partial D}g\frac{\partial z}{\partial n}\mathrm{d}s \quad\text{ for all }z\in X(\omega_i).
\end{align}
The remaining of this section is devoted to proving the well posedness of the nonstandard variational formulation \eqref{eq:nonstd-variational} and deriving the {\it a priori} error estimate. Our final result is the following,
\begin{theorem}\label{lem:very-weak}
Given $g\in L^2(\partial\omega_i\backslash\partial D)$. Let $v$ be the solution to \noteLg{\eqref{eq:nonstd-variational}}, then there exists a constant ${\rm C}_{{\rm weak}}$ independent of the parameter $\epsilon$ such that
\[
\normL{v}{\omega_i}\leq {\rm C}_{{\rm weak}}\textcolor{black}{(\PeH^{-1/2}+\PeH^{1/2})H^{1/2}}\|g\|_{L^2(\partial \omega_i\backslash\partial D)}.
\]
\end{theorem}
To prove \cref{lem:very-weak}, one has to first derive the $L^2(\partial\omega_i)$-estimate of the normal trace $\frac{\partial z}{\partial n}$ for any $z\in X(\omega_i)$. This is established in the following theorem:

\begin{theorem}\label{thm:pw-Regularity}
Let $w\in L^2(\omega_i)$ and let $z\in X(\omega_i)$ satisfy
\begin{equation}\label{eq:pde-dual}
\left\{\begin{aligned}
\mathcal{L}_i^* (z):=-\epsilon\Delta z-\bm{b}\cdot\nabla z&=w && \text{ in } \omega_i,\\
z&=0 &&\text{ on }\partial \omega_i.
\end{aligned}\right.
\end{equation}
Then for some constant ${\rm C}_{{\rm weak}} $ independent of the parameter $\epsilon$ and $H$, there holds
\begin{align*}
\|\frac{\partial z}{\partial n}\|_{L^2(\partial\omega_i)}
&\leq {\rm C}_{{\rm weak}}\textcolor{black}{(1+\PeH)}\left(\|\bm{b}\|_{L^{\infty}(D)}\epsilon\right)^{-1/2}\normL{w}{\omega_i}.
\end{align*}
This constant ${\rm C}_{{\rm weak}}$ can change value from context to context.
\end{theorem}
\begin{proof}

Note that $z\in X(\omega_i)$ is the unique solution to the following weak formulation
\begin{align*}
\forall q\in H^1_{0}(\omega_i): \int_{\omega_i}\mathcal{L}_i^{*}(z) q\;\dx=\int_{\omega_i}wq\dx.
\end{align*}
Taking $q:=z$ and applying integration by parts, we arrive at
\begin{align*}
\epsilon\normL{\nabla z}{\omega_i}^2=\int_{\omega_i}wz\dx.
\end{align*}
Since $z|_{\partial \omega_i}=0$, then an application of the Poincar\'{e} inequality leads to
\begin{align}\label{est:1}
\normL{\nabla z}{\omega_i}\lesssim \noteLi{\frac{H}{\epsilon}}\normL{w}{\omega_i}.
\end{align}
Recall that the mesh P\'{e}clet number $\PeH$ of the coarse mesh $\mc{T}^H$ is
$\PeH:=H\|b\|_{L^{\infty}(D)}/\epsilon$. Consequently, plugging this estimate into \eqref{eq:pde-dual} results in
\begin{align*}
\normL{\Delta z}{\omega_i}\lesssim \epsilon^{-1}(1+\PeH)\normL{w}{\omega_i}.
\end{align*}

By the so-called Miranda-Talenti estimate on a convex domain \cite{MR0775683}, we can obtain the following {\it a priori} estimate
\begin{align}\label{est-delta}
\noteLg{|{z}|_{H^2(\omega_i)}}&\leq \normL{\Delta z}{\omega_i}
\lesssim  \epsilon^{-1}(1+\PeH)\normL{w}{\omega_i} .
\end{align}
\noteLg{
Note that 
\begin{align*}
\frac{\partial z}{\partial n}=\nabla z\cdot n.
\end{align*}
Next, we invoke a quantitative trace theorem on a bounded open domain with a Lipschitz boundary \cite[Theorem 1.5.1.10]{MR0775683} to obtain an estimate of $\left\|\frac{\partial z}{\partial n}\right\|_{L^2(\partial\omega_i)}$, which states, 
\begin{align*}
\left\|\frac{\partial z}{\partial n}\right\|_{L^2(\partial\omega_i)}&\lesssim 
 \delta^{1/2}H^{1/2}|{z}|_{H^2(\omega_i)}+
\delta^{-1/2}H^{-1/2}\|\nabla z\|_{L^{2}(\omega_i)} \text{ for any }\delta\in (0,1)
\end{align*}
with the hidden constant independent of the size of the domain $\omega_i$. 

Since we are interested in the case with $\epsilon\ll H$, then taking $\delta:=\PeH^{-1}$ and combining with \eqref{est:1} and \eqref{est-delta}, we derive the desired assertion.
}
\end{proof}
Thanks to this {\it a priori} estimate presented in Theorem \ref{thm:pw-Regularity}, we are ready to state the well posedness of the nonstandard variational formulation, cf. \eqref{eq:nonstd-variational},
\begin{lemma}
Let $v$ be the solution to problem \noteLg{\eqref{eq:nonstd-variational}} and let the test space $X(\omega_i)$ be defined in \eqref{eq:test-space}. Then the nonstandard variational form \eqref{eq:nonstd-variational} is well posed.
\end{lemma}
\begin{proof}
To prove the well-posedness of the nonstandard variational form \eqref{eq:nonstd-variational}, we introduce a bilinear form
$c(\cdot,\cdot)$ on $L^2(\omega_i)\times L^2(\omega_i)$ and a linear form $b(\cdot)$ on $L^2(\omega_i)$, defined by
\begin{align*}
c(w_1,w_2)&:=\int_{\omega_i}w_1 w_2\;\dx\quad\text{ for all } w_1, w_2\in L^2(\omega_i)\\
r(w)&:=-\epsilon\int_{\partial\omega_i\backslash\partial D}\noteLi{\frac{\partial z(w)}{\partial n}} \noteLi{g}\;\mathrm{d}s\quad
\text{ for all }w\in L^2(\omega_i),
\end{align*}
with \noteLi{$z(w)$} being the unique solution to \eqref{eq:pde-dual}.
It follows from Theorem \ref{thm:pw-Regularity} \textcolor{black}{and the definition of the mesh P\'{e}clet number $\PeH$} that
\begin{align}
\|r\|:=\sup_{w\in L^2(\omega_i)}\frac{r(w)}{\|w\|_{L^2(\omega_i)}}
&\leq {\rm C}_{{\rm weak}}\textcolor{black}{(1+\PeH)}\sqrt{\frac{\epsilon}{\|b\|_{L^{\infty}(D)}}}\normL{g}{\partial\omega_i\backslash\partial D}\nonumber\\
&\leq {\rm C}_{{\rm weak}}\textcolor{black}{(\PeH^{-1/2}+\PeH^{1/2})H^{1/2}}\normL{g}{\partial\omega_i\backslash\partial D}.\label{eq:b}
\end{align}
This yields well-posedness of the following variational problem: find $v\in L^2(\omega_i)$ such that
\begin{align}\label{eq:variation2}
c(v,w)=r(w)\quad\text{ for all }w\in L^2(\omega_i).
\end{align}
The equivalence of problems \eqref{eq:variation2} and \eqref{eq:nonstd-variational} implies the desired well-posedness of \eqref{eq:nonstd-variational}.
\end{proof}
Finally, we are ready to prove the main result stated in \cref{lem:very-weak}:
\begin{proof}[Proof of \cref{lem:very-weak}]
For all $w\in L^2(\omega_i)$, we obtain from \eqref{eq:b} and \eqref{eq:variation2}:
\begin{align*}
\int_{\omega_i} v w\;\dx:=c(v,w)&=r(w)\\
&\leq \text{C}_{\text{weak}}(\PeH^{-1/2}+\PeH^{1/2})H^{1/2}\normL{w}{\omega_i}\|g\|_{L^2(\partial\omega_i\backslash\partial D)}.
\end{align*}
Then this completes the proof by taking $w:=v$.
\end{proof}
Furthermore, we present the {\it a priori} estimate in energy norm to facilitate the proofs in \cref{sec:error}. This estimate is standard and is commonly known as the Cacciappoli inquality, which can be found in many literature, e.g., \cite{MR2801210,GL18}.
\begin{corollary}\label{corollary:very-weak}
Let $g\in L^2(\partial\omega_i\backslash\partial D)$ and let $v$ be the solution to \eqref{eq:nonstd-variational}. Then there exists a constant ${\rm C}_{{\rm weak}}$ independent of the parameter $\epsilon$ such that
\[
\normL{\chi_i\nabla v}{\omega_i}\leq {\rm C}_{{\rm weak}}\left(\textcolor{black}{\PeH}+\PeH^{-1/2}\right)H^{-1/2}\|g\|_{L^2(\partial \omega_i\backslash\partial D)}.
\]
\end{corollary}
\begin{proof}
Recall that $\{\chi_i\}_{i=1}^N$ is a $(\Cov,C_{\infty},C_{\text{G}})$ partition of unity subordinate to the cover $\{\omega_i\}_{i=1}^N$, which implies that $\chi_i=0$ on $\partial\omega_i\backslash\partial D$. Multiplying \eqref{eq:pde-very} by $\chi_i^2v$ and applying integration by parts, we arrive at
\begin{align*}
\epsilon\int_{\omega_i}\chi_i^2|\nabla v|^2\dx=-2\epsilon \int_{\omega_i}\nabla v\cdot \nabla\chi_i\chi_i v\dx
+\int_{\omega_i}\bm{b}\cdot\nabla\chi_i\chi_i v^2\dx.
\end{align*}
Then an application of the Young's inequality leads to
\begin{align*}
\int_{\omega_i}\chi_i^2|\nabla v|^2\dx
\lesssim \left(H^{-2}+\frac{\|\bm{b}\|_{L^{\infty}(D)}}{\epsilon H}\right)\|v\|_{L^2(\omega_i)}^2
.
\end{align*}
Combining with \cref{lem:very-weak}, we obtain 
\begin{align*}
\int_{\omega_i}\chi_i^2|\nabla v|^2\dx
&\lesssim \textcolor{black}{(1+\PeH)^2}\left(H^{-2}+\frac{\|\bm{b}\|_{L^{\infty}(D)}}{\epsilon H}\right)
\frac{\epsilon}{\|\bm{b}\|_{L^{\infty}(D)}}
\|g\|_{L^2(\partial\omega_i\backslash\partial D)}^2\\
&\lesssim H^{-1}\textcolor{black}{\left(\PeH^2+\PeH^{-1}\right)}
\|g\|_{L^2(\partial\omega_i\backslash\partial D)}^2
.
\end{align*}
Then the desired assertion follows after taking the square root, and this completes the proof.
\end{proof}

\section*{Acknowledgments}
We thank the anonymous referees for very detailed and constructive comments.

\bibliographystyle{siamplain}
\bibliography{references}
\end{document}